\documentclass[a4paper,11pt,reqno]{amsart}
\usepackage[latin1]{inputenc}
\usepackage[english]{babel}
\usepackage{amssymb,amsfonts, amsmath, color}
\usepackage[margin=2.7cm]{geometry}
\usepackage{xcolor}
\usepackage{graphicx, color, enumerate}
\usepackage[latin1]{inputenc}
\usepackage[active]{srcltx}
\usepackage{tikz}
\usepackage{pgf}
\usepackage{etex}
\usepackage{verbatim}
\usepackage{tikz-3dplot}
\usepackage{amsmath}
\usepackage{amsfonts}
\usepackage{amssymb}
\usepackage{amsthm}
\usepackage{algpseudocode}
\usepackage{float}
\usepackage{xcolor}
\usepackage{mathrsfs}
\usepackage{import}
\usepackage{geometry}
\usepackage{fancyhdr}
\usepackage{enumitem}
\usepackage{fp}

\newtheorem{theorem}{Theorem}[section]
\newtheorem{proposition}[theorem]{Proposition}
\newtheorem{corollary}[theorem]{Corollary}
\newtheorem{lemma}[theorem]{Lemma}
\newtheorem{remark}[theorem]{Remark}
\newtheorem{definition}[theorem]{Definition}
\newtheorem{example}[theorem]{Example}

\usepackage[colorlinks, citecolor=blue, linkcolor=red]{hyperref}

\newcommand{\bcl}{\begin{center}}
\newcommand{\ecl}{\end{center}}
\newcommand{\brl}{\begin{right}}
\newcommand{\erl}{\end{right}}
\newcommand{\ben}{\begin{enumerate}}
\newcommand{\een}{\end{enumerate}}
\newcommand{\overliner}{\begin{array}}
\newcommand{\earr}{\end{array}}
\newcommand{\btab}{\begin{tabular}}
\newcommand{\etab}{\end{tabular}}
\newcommand{\bdoc}{\begin{document}}
\newcommand{\edoc}{\end{document}}
\newcommand{\beqy}{\begin{eqnarray}}
\newcommand{\eeqy}{\end{eqnarray}}

\newcommand{\beqi}{\begin{eqnarray*}}
\newcommand{\eeqi}{\end{eqnarray*}}
\newcommand{\bitem}{\begin{itemize}}

\newcommand{\eitem}{\end{itemize}}
\newcommand{\nln}{\newline}
\newcommand{\newt}{\newtheorem}

\newcommand{\pa}{\partial}
\newcommand{\re}{{I\!\!R}}
\newcommand{\Rn}{\R^N}
\newcommand{\xr}{x\in\R }
\newcommand{\x}{\times}
\newcommand{\dyle}{\displaystyle}
\newcommand{\ene}{{I\!\!N}}
\newcommand{\irn}{\int\limits_{\R^N}}
\newcommand{\io}{\int\limits_{\O}}
\newcommand{\meas}{{\rm meas\,}}
\newcommand{\dif}{\nabla_{xy}}
\newcommand{\sign}{{\rm sign}}
\newcommand{\map}{\longrightarrow }
\newcommand{\imp}{\Longrightarrow }
\renewcommand{\div}{\nabla\cdot }
\newcommand{\sen}{{\rm sen\,}}
\newcommand{\tg}{{\rm tg\,}}
\newcommand{\arcsen}{{\rm arcsen\,}}
\newcommand{\arctg}{{\rm arctg\,}}
\newcommand{\supp}{{\textsl supp\ }}
\newcommand{\ity}{\int_{-\iy}^{+\iy}}
\newcommand{\limit}{\lim\limits}
\newcommand{\limi}{\limit_{n\to\infty}}
\newcommand{\sumi}{\sum\limits_{n=1}^{\infty}}
\newcommand{\ulu}{\underline u}
\newcommand{\ulw}{\underline w}
\newcommand{\ulz}{\underline z}
\newcommand{\ulv}{\underline v}
\newcommand{\uls}{\underline s}
\newcommand{\olu}{\overline u}
\newcommand{\olv}{\overline v}
\newcommand{\ols}{\overline s}
\newcommand{\ob}{\overline\b}
\newcommand{\ovar}{\overline\var}
\newcommand{\wv}{\widetilde v}
\newcommand{\wu}{\widetilde u}
\newcommand{\ws}{\widetilde s}
\renewcommand{\a }{\alpha }
\renewcommand{\b }{\beta }
\newcommand{\g }{\gamma}
\newcommand{\G }{\Gamma }
\renewcommand{\d }{\delta }

\newcommand{\D }{\Delta }
\newcommand{\e }{\varepsilon }
\newcommand{\z }{\zeta }
\renewcommand{\l }{\lambda }
\renewcommand{\L }{\Lambda }
\newcommand{\m }{\mu }
\newcommand{\n }{\nabla }
\newcommand{\s }{\sigma }
\newcommand{\Sig }{\Sigma }
\renewcommand{\t }{\tau }
\newcommand{\var }{\varphi }
\renewcommand{\o }{\omega }
\renewcommand{\O }{\Omega }
\newcommand{\R}{{\mathbb{R}}}
\newcommand{\bC}{{\bf C}}
\newcommand{\bZ}{{\bf Z}}
\newcommand{\bN}{{\bf N}}
\newcommand{\bQ}{{\bf Q}}
\newcommand{\bK}{{\bf K}}
\newcommand{\bI}{{\bf I}}
\newcommand{\bv}{{\bf v}}
\newcommand{\bV}{{\bf V}}
\DeclareMathOperator{\suppo}{supp} \DeclareMathOperator{\di}{div}

\newenvironment{Proof}{\Rmovelastskip\vskip12pt
plus 1pt \noindent\em\rm}{\hfill {\qed \hskip .2cm}}

\begin{document}

\title[]{Semigroup Criteria for Blow-up and Global Existence of Semilinear Heat Equations on Metric Measure Spaces}

\author{Giulia Meglioli}

\author{Fabio Punzo}

\address{\hbox{\parbox{5.7in}{\medskip \noindent{Giulia Meglioli, \\Dipartimento di Matematica, \\Politecnico di Milano, \\Piazza Leonardo da Vinci 32, 20133, Milano, Italy \\ [3pt] \emph{E-mail address: }{\tt giulia.meglioli@polimi.it}}}}}

\address{\hbox{\parbox{5.7in}{\medskip \noindent{Fabio Punzo, \\Dipartimento di Matematica, \\Politecnico di Milano, \\Piazza Leonardo da Vinci 32, 20133, Milano, Italy \\ [3pt] \emph{E-mail address: }{\tt fabio.punzo@polimi.it}}}}}


\subjclass[2010]{}

\maketitle

\begin{abstract}
We study finite-time blow-up and global existence for semilinear heat equations with time-dependent reaction coefficients on metric measure spaces. For convex nonlinearities, we establish a blow-up criterion that retains the dependence on the initial datum through its linear evolution and does not require stochastic completeness. Under additional conservativity and heat kernel assumptions, we also prove blow-up at the critical Fujita exponent associated with polynomial volume growth. Our main global existence result is based on a heat-kernel-weighted functional space in which the mild solution map becomes a contraction. This yields a unified construction covering both polynomial and exponential heat kernel decay, general locally Lipschitz nonlinearities, and time-dependent reaction coefficients, without convexity or monotonicity assumptions. In particular, the global existence criterion only uses the behaviour of the nonlinearity on the range explored by the solution. Applications include Riemannian manifolds, metric graphs, and, within a separate discrete framework, weighted graphs. The arguments, especially the functional construction underlying global existence, are new and lead to new results even for the classical underlying spaces considered above.
\end{abstract}

\section{Introduction}\label{sec0}
We investigate global existence and finite-time blow-up of nonnegative solutions to semilinear parabolic problems of the form
\begin{equation}\label{problema}
\begin{cases}
u_t=\Delta u+h(t)f(u) & \text{in }X\times(0,T),\\
u=u_0 & \text{in }X\times\{0\},
\end{cases}
\end{equation}
where $(X,d,\mu)$ is a metric measure space equipped with a regular symmetric Dirichlet form whose associated sub-Markovian semigroup admits a heat kernel. We write $\Delta=-L$, where $L$ is the nonnegative self-adjoint operator associated with the form, so that the heat semigroup is $P_t=e^{t\Delta}=e^{-tL}$. Here $T\in(0,\infty]$, the function $h\in C([0,\infty))$ is positive for $t>0$, and $f:[0,\infty)\to[0,\infty)$ is locally Lipschitz, with $f(0)=0$ and $f(v)>0$ for $v>0$. The initial datum $u_0$ is bounded, continuous, nonnegative, and nontrivial. Further assumptions on $f$ are specified separately for the blow-up and global existence results. The continuous framework includes Riemannian manifolds and metric graphs. Discrete weighted graphs are treated separately, using the same semigroup arguments whenever the required kernel assumptions hold.

The interplay between diffusion and reaction has been extensively studied in the Euclidean setting. For $h\equiv1$ and $f(u)=u^p$, the classical Fujita phenomenon identifies the exponent $p_F=1+2/n$ as the threshold between finite-time blow-up for every nontrivial nonnegative initial datum and global existence for sufficiently small initial data on $\mathbb R^n$ (see \cite{Fujita, H, KST, Sug}). We refer to \cite{BB, DL, Levine, QS} for surveys and further developments. On Riemannian manifolds, geometry can substantially change this picture; see, for instance, \cite{BPT, GrilloMeglioliPunzo2023, MaMoPu, Punzo, Pu22, Zhang}. In particular, on suitable manifolds with positive bottom of the spectrum, small-data global existence holds for every power $p>1$, and spectral quantities also govern the behaviour of more general nonlinearities (see \cite{Punzo}).

Related questions have received considerable attention on infinite graphs. Results on global existence and blow-up for semilinear heat equations can be found in \cite{GrilloMeglioliPunzoGraphs, LinWu1, LinWu2, PuZu1, PuZu2, Wu}, while nonexistence results obtained through geometric and space-time volume conditions are developed in \cite{MoPuSo}.

Metric measure spaces provide a natural framework for studying diffusion in both smooth and nonsmooth environments. The construction of heat flows, estimates for their kernels, and the analysis of associated parabolic equations have been investigated extensively; see, e.g., \cite{AGS2, AGS, AMS, CGrigL, CavMond, DePMS, DePS, GrigHK, GrigHuLau, GrigKum, GrigHH3, GrigHH, GrigHH2}. A central theme is the relationship between the geometry of the underlying space, functional inequalities, and the behaviour of the heat semigroup. For nonlinear equations, heat kernel estimates also provide a way to quantify the competition between diffusion and reaction without relying on local coordinates or a differential description of the underlying space. Blow-up and global existence for a model semilinear parabolic equation on metric measure spaces were studied in \cite{PT}, with applications to metric graphs and Riemannian manifolds. The present paper develops criteria for general locally Lipschitz nonlinearities and time-dependent reaction coefficients, under the respective assumptions needed for blow-up and global existence.

Our aim is not only to extend individual existence and nonexistence results to a broader class of spaces, but also to formulate the competition between diffusion and reaction directly in terms of the heat semigroup. Two complementary mechanisms emerge. On the blow-up side, the relevant quantity is the linear evolution $P_Tu_0$ of the entire initial datum at a fixed time. On the global existence side, the heat kernel itself provides a natural space-time profile in which the nonlinear mild solution map can be controlled. The latter leads to a single weighted functional framework that applies to both polynomial and exponential heat kernel decay and separates the geometric input from the local behaviour of the nonlinearity.

In \cite{LSZ}, the autonomous problem corresponding to $h\equiv1$ is studied in the abstract framework of $L^p$-valued mild solutions, with $1\le p<\infty$. The authors establish local existence and a criterion for nonexistence of global solutions for sub-Markovian semigroups, without requiring an integral kernel. Their work is concerned with blow-up, while they do not address global existence. Our nonexistence result allows a continuous time-dependent coefficient $h$ and is formulated for mild solutions that are bounded on compact time intervals, without requiring the initial datum to belong to an $L^p$ space with $p<\infty$. A detailed comparison in the autonomous case is given in Remark \ref{rem:comparison-LSZ}.

For time-dependent power reactions on infinite graphs with positive spectral bottom, in \cite{PuSa}, the authors prove small-data global existence by a contraction argument in a space controlled by a shifted heat kernel column. Our global existence criterion extends this construction by introducing an additional time-dependent weight and replacing the power estimates by local Lipschitz bounds for a general nonlinearity. More importantly, the resulting weighted heat kernel space is not tied to a positive spectral bottom: the same functional construction applies under either polynomial or exponential heat kernel decay, within the respective continuous and discrete semigroup settings.

Our blow-up result (see Theorem \ref{teo1}) gives a sufficient condition that retains the dependence on the entire initial datum through its linear evolution. Assume that $f$ is nondecreasing and convex and satisfies condition
\[
\int_1^\infty\frac{ds}{f(s)}<\infty.
\]
Set
\[
J(a):=\int_a^\infty\frac{ds}{f(s)}, \qquad a>0, \qquad H(t):=\int_0^t h(s)\,ds.
\]
Under these assumptions, $J$ is a strictly decreasing bijection from $(0,\infty)$ onto $(0,\infty)$. Theorem \ref{teo1} shows that, if
\[
\|P_Tu_0\|_{L^\infty(X,\mu)} \ge J^{-1}(H(T))
\]
for some $T>0$, then the maximal solution blows up no later than $T$. The quantity $P_Tu_0$ describes the initial datum after diffusion alone has acted for time $T$, whereas $J^{-1}(H(T))$ is the threshold determined by the reaction and its accumulated temporal strength. The condition is imposed at a single time and requires no asymptotic heat kernel estimate. In particular, it detects blow-up for sufficiently large multiples of a fixed initial profile, including situations in which a criterion based only on long-time decay gives no conclusion.

The proof studies the backward semigroup transform $P_{T-t}u(t)$ and derives a scalar differential inequality by convexity. This approach is related to the arguments used in \cite{GrilloMeglioliPunzo2023, GrilloMeglioliPunzoGraphs}, but the present formulation improves their scope in several directions. Stochastic completeness, imposed in those blow-up arguments, is unnecessary here: sub-Markovianity and $f(0)=0$ suffice for the required Jensen inequality. Moreover, in the common setting, the earlier conditions that ensure blow-up independently of the size of the nontrivial datum yield sufficient instances of the present criterion, whereas our condition also records its amplitude and spatial distribution explicitly. The proof includes equality at the threshold, even when the essential supremum of $P_Tu_0$ is not attained. It does so by establishing a uniform positive gap in the necessary condition for a solution bounded up to time $T$.

For suitable conservative spaces, the necessary condition can also be strengthened to treat a critical power (see Theorem \ref{thm:critical-volume}). More precisely, when $h\equiv1$, the heat kernel satisfies $(HK_\beta)$, and the volume of large balls satisfies $V(x_0,R)\le CR^D$, we prove blow-up at
\[
p=1+\frac{\beta}{D}
\]
for $f(u)=u^p$. The argument applies the necessary condition to solutions restarted at arbitrary times. Global existence would force a uniform bound on their total mass, while the lower heat kernel estimate forces the reaction to produce at least logarithmic growth of that mass. This yields a contradiction. The result includes the critical Fujita exponent on $\mathbb R^n$, its counterpart on manifolds with nonnegative Ricci curvature, and the exponent determined by the large-scale dimension of flat cylinders. This treatment of the polynomial critical regime goes beyond the noncritical heat kernel applications in \cite{LSZ} and complements the positive-spectral-bottom setting of the graph results discussed above.

The main contribution on the global existence side is a functional construction based on a weighted heat kernel profile. Given a pole $y_0$, a positive shift $\gamma$, and a nondecreasing continuous weight $\omega\ge1$, we consider nonnegative functions satisfying
\[
u(x,t)\le M\omega(t)K_{t+\gamma}(x,y_0).
\]
The distance between two functions is measured by dividing their difference by the same space-time profile and taking an essential supremum. Thus the functional space is adapted directly to the linear diffusion: the shifted heat kernel incorporates its spatial behaviour, while the factor $\omega$ compensates for the temporal effect of the reaction. The semigroup identity then reduces both invariance of this class and contraction of the mild solution map to scalar integral estimates.

A significant feature of this construction is that it is not tied to a particular large-time regime of the heat kernel. Taking $\omega\equiv1$ allows the same argument to cover $\mathbb R^n$ and manifolds with polynomial heat kernel decay, while suitable nonconstant weights accommodate nonlinearities with a linear part when exponential decay is available. Polynomial and spectral regimes are therefore treated by one contraction mechanism rather than by separate existence arguments. Moreover, since the nonlinear estimates are expressed through local Lipschitz constants on the range allowed by the weighted profile, the construction does not require convexity or monotonicity of $f$ and may depend only on its behaviour near the origin. In this sense, the weighted space provides a direct link between the decay properties of the linear semigroup and the amount of nonlinear growth that can be admitted. The arguments developed here, and in particular the functional construction underlying the global existence theory, are new. They lead to results that appear to be new even for the classical underlying spaces considered below, including Riemannian manifolds and graph settings.

Theorem \ref{teo2} makes this principle precise. It gives an explicit smallness condition on the initial datum through domination by a heat kernel column and yields a global bound by the corresponding weighted profile. Neither convexity nor monotonicity of $f$ is required. When the construction keeps the solution in a fixed interval $[0,\delta_0]$, the criterion depends only on the behaviour of $f$ in that interval, with no restriction on its growth beyond $\delta_0$. Uniqueness holds in the full class of nonnegative mild solutions bounded on compact time intervals, not merely within the weighted class used for the construction.

For the power nonlinearity $f(u)=u^p$, a basic consequence is the sufficient condition
\[
\int_0^\infty h(t)\overline{\mathcal K}(t+\gamma)^{p-1}\,dt<\infty,
\]
where $\overline{\mathcal K}(t)$ denotes the essential supremum of the heat kernel over the spatial variables. This condition directly incorporates both the temporal coefficient and the rate of heat kernel decay. Polynomial decay yields Fujita-type ranges, while a positive spectral bottom, together with a uniform heat kernel bound at a positive time, yields small-data global existence for every $p>1$ when $h$ is bounded or grows at most polynomially. For exponentially increasing $h$, the same condition expresses the balance between the growth rate of the reaction coefficient and the spectral decay. These applications illustrate how the blow-up and global existence criteria retain their form across substantially different geometric settings.

\medskip

The paper is organized as follows. In Section \ref{mf}, we introduce the metric measure framework, the Laplace operator, and the associated heat kernel. We also recall heat kernel estimates in several settings and distinguish metric graphs from discrete weighted graphs. In Section \ref{statements}, we state our main results on finite-time blow-up and global existence. Sections \ref{blowup-proof} and \ref{globalex-proof} contain their proofs. In Section \ref{model}, we specialize the general criteria to power nonlinearities, establish the critical blow-up result, and discuss the role of the time coefficient and of the underlying geometry. Section \ref{model-beyond-powers} treats nonlinearities beyond powers, including sums of powers and exponential nonlinearities. Appendix \ref{app:kernel-background} collects additional background on heat kernel estimates.
\section{Mathematical framework}\label{mf}

\setcounter{equation}{0}

\subsection{Metric measure spaces and Dirichlet forms}\label{dirichletforms}

We denote by $\mathcal B(X)$ the Borel $\sigma$-algebra of a metric space $X$.

\begin{definition}\label{MMS}
Except in the discussion of discrete weighted graphs, throughout the paper
we work with metric measure spaces $(X,d,\mu)$ satisfying the following
assumptions:
\begin{itemize}
\item $(X,d)$ is connected, locally compact, separable, and complete;
\item $d$ is a geodesic metric;
\item $\mu$ is a locally finite Radon measure on $\mathcal B(X)$ with full support:
\[
\operatorname{supp}\mu=X.
\]
\end{itemize}
\end{definition}

For $x\in X$ and $r>0$, we write
\[
B_r(x):=\{y\in X:d(x,y)<r\}.
\]

Here, the geodesic assumption means that any two points can be joined by a curve
whose length equals their distance. Under the assumptions in Definition \ref{MMS},
closed bounded sets are compact. Consequently,
\[
0<\mu(B_r(x))<\infty
\qquad \text{for every }x\in X\text{ and }r>0.
\]

Let $(\mathcal E,\mathcal F)$ be a symmetric Dirichlet form on the real Hilbert space
$L^2(X,\mu)$. Thus, $\mathcal F$ is a dense linear subspace of $L^2(X,\mu)$ and
$\mathcal E$ is a symmetric, nonnegative bilinear form on $\mathcal F$ which is
closed and Markovian. The latter property means that
\[
\varphi(u)\in\mathcal F,
\qquad
\mathcal E(\varphi(u),\varphi(u))\le \mathcal E(u,u)
\]
for every $u\in\mathcal F$ and every normal contraction $\varphi:\mathbb R\to\mathbb R$.
Here, a normal contraction satisfies
\[
\varphi(0)=0,
\qquad
|\varphi(a)-\varphi(b)|\le |a-b|
\quad \text{for all }a,b\in\mathbb R.
\]
We refer to \cite[Section 1.1]{FOT} for further details.

We will always assume that $(\mathcal E,\mathcal F)$ is regular, in the following sense.

\begin{definition}\label{regularity}
A Dirichlet form $(\mathcal E,\mathcal F)$ is called \emph{regular} if there exists a set
\[
\mathcal C\subseteq\mathcal F\cap C_c(X)
\]
which is dense in $\mathcal F$ with respect to the $\mathcal E_1$-norm and dense in
$C_c(X)$ with respect to the supremum norm, where
\[
\mathcal E_1(u,v):=\mathcal E(u,v)+\langle u,v\rangle_{L^2(X,\mu)},
\qquad
\|u\|_{\mathcal E_1}^2=\mathcal E_1(u,u).
\]
Such a set $\mathcal C$ is called a core of the regular form.
\end{definition}

Equivalently, the form is regular if $\mathcal F\cap C_c(X)$ itself has these density
properties. To avoid ambiguity, a subset of $\mathcal F$ which is dense only with respect
to the $\mathcal E_1$-norm will be called a \emph{form core}.

The form may be local or non-local.

For functions in $L^2(X,\mu)$, supports are understood as essential
supports: $\operatorname{supp}u$ is the complement of the largest open
set on which $u=0$ $\mu$-almost everywhere.
Likewise, a function is said to be constant on an open set if it is
constant $\mu$-almost everywhere on that set.

\begin{definition}\label{locality}
A Dirichlet form $(\mathcal E,\mathcal F)$ is called:
\begin{itemize}
\item \emph{local} if
\[
\mathcal E(u,v)=0
\]
whenever $u,v\in\mathcal F$ have compact disjoint supports;
\item \emph{strongly local} if
\[
\mathcal E(u,v)=0
\]
whenever $u,v\in\mathcal F$, $v$ has compact support, and $u$ is constant in an open
neighborhood of $\operatorname{supp}v$.
\end{itemize}
\end{definition}

\begin{proposition}\label{prop22}
Let $(X,d,\mu)$ be as in Definition \ref{MMS}, and let $(\mathcal E,\mathcal F)$ be a regular
symmetric Dirichlet form on $L^2(X,\mu)$. There exists a unique nonnegative self-adjoint
operator $L_{\mathcal E}$ with domain $D(L_{\mathcal E})$ such that:
\begin{enumerate}[label=(\roman*)]
\item $D(L_{\mathcal E})\subseteq\mathcal F$ and
\[
\mathcal E(u,v)=\int_X (L_{\mathcal E}u)v\,d\mu
\]
for all $u\in D(L_{\mathcal E})$ and $v\in\mathcal F$;

\item $D(L_{\mathcal E})$ is dense in $\mathcal F$ with respect to the $\mathcal E_1$-norm;

\item if $u\in\mathcal F$ and $w\in L^2(X,\mu)$ satisfy
\[
\mathcal E(u,v)=\int_X wv\,d\mu
\]
for every $v$ in a form core, then
\[
u\in D(L_{\mathcal E}),
\qquad L_{\mathcal E}u=w.
\]
\end{enumerate}
\end{proposition}

We refer to \cite[Corollary 1.3.1]{FOT} and
\cite[Theorem VI.2.1]{Kato} for these facts.
For part (iii), the identity extends from a form core to all of
$\mathcal F$, since both sides are continuous with respect to the
$\mathcal E_1$-norm.

\begin{definition}\label{def:laplacian}
The Laplacian associated with $\mathcal E$ is
\[
\Delta:=-L_{\mathcal E},
\qquad D(\Delta):=D(L_{\mathcal E}).
\]
In particular,
\[
\int_X(\Delta u)u\,d\mu=-\mathcal E(u,u)\le0
\qquad \text{for every }u\in D(\Delta).
\]
\end{definition}

\subsection{Semigroups and heat kernels}\label{heats}

We use the following definition of heat kernel; see, for instance,
\cite[Section 1]{GrigHuLau}.

\begin{definition}[Heat kernel]\label{def:heatkernel}
A nonnegative measurable function
\[
K:X\times X\times(0,\infty)\to[0,\infty),
\qquad K_t(x,y):=K(x,y,t),
\]
is called a \emph{heat kernel} if it satisfies the following properties:
\begin{enumerate}[label=(\alph*)]
\item for every $t>0$,
\[
\int_X K_t(x,y)\,d\mu(y)\le1
\qquad \text{for $\mu$-almost every }x\in X;
\]

\item for every $t>0$,
\[
K_t(x,y)=K_t(y,x)
\qquad \text{for $\mu\otimes\mu$-almost every }(x,y);
\]

\item for every $s,t>0$,
\begin{equation}\label{eq10}
\int_X K_t(x,y)K_s(y,z)\,d\mu(y)=K_{t+s}(x,z)
\end{equation}
for $\mu\otimes\mu$-almost every $(x,z)$;

\item for every $u\in L^2(X,\mu)$,
\[
\lim_{t\downarrow0}
\left\|\int_X K_t(\cdot,y)u(y)\,d\mu(y)-u\right\|_{L^2(X,\mu)}=0.
\]
\end{enumerate}
\end{definition}

If, in addition,
\begin{equation}\label{def:stochastic}
\int_X K_t(x,y)\,d\mu(y)=1
\end{equation}
for every $t>0$ and $\mu$-almost every $x\in X$, then the heat kernel is called
\emph{stochastically complete}, or \emph{conservative}.

Conservativity is a property of the chosen heat kernel, or equivalently
of its associated semigroup.
Once the Dirichlet form is fixed, we also say that the corresponding
Dirichlet space is \emph{stochastically complete}.
For brevity, we may say that $X$ is stochastically complete, always
with reference to this fixed form.

\begin{definition}[Heat semigroup]\label{heatsemigroup}
The semigroup associated with a heat kernel is defined by $P_0=\operatorname{id}$ and
\[
P_tu(x)=\int_X K_t(x,y)u(y)\,d\mu(y),
\qquad t>0.
\]
\end{definition}

The operators $\{P_t\}_{t\ge0}$ form a strongly continuous, self-adjoint contraction
semigroup on $L^2(X,\mu)$. They are positivity preserving and sub-Markovian:
\[
0\le u\le1
\quad\Longrightarrow\quad
0\le P_tu\le1.
\]
The term \emph{Markovian} is also commonly used for this property.
Conservativity means that $P_t1=1$.

The semigroup extends consistently to contraction semigroups on $L^p(X,\mu)$
for every $p\in[1,\infty]$. In particular,
\begin{equation}\label{eq01}
P_tu\text{ has the same meaning on }L^p(X,\mu)\cap L^q(X,\mu),
\qquad p,q\in[1,\infty].
\end{equation}
These extensions are strongly continuous for $1\le p<\infty$.
Strong continuity on $L^\infty(X,\mu)$ does not hold in general.

\begin{definition}[Nonnegative generator]\label{generator}
The nonnegative generator $\mathcal L$ of $\{P_t\}_{t\ge0}$ is defined by
\begin{equation}\label{eq03}
\mathcal Lu:=\lim_{t\downarrow0}\frac{u-P_tu}{t},
\end{equation}
where the limit is taken in $L^2(X,\mu)$. Its domain consists of the functions
$u\in L^2(X,\mu)$ for which this limit exists.
\end{definition}

The operator $\mathcal L$ is densely defined, nonnegative, and self-adjoint, and
\[
P_t=e^{-t\mathcal L}.
\]
Thus, the infinitesimal generator in the usual semigroup convention is $-\mathcal L$.

Conversely, the regular symmetric Dirichlet form $(\mathcal E,\mathcal F)$
determines the semigroup
\[
P_t=e^{-tL_{\mathcal E}}=e^{t\Delta},
\qquad t\ge0.
\]
This is a strongly continuous, self-adjoint, sub-Markovian contraction semigroup
on $L^2(X,\mu)$. Its existence does not, by itself, imply the existence of an
integral kernel with respect to $\mu$.

Whenever a heat kernel is used below, we assume that it represents this semigroup:
\[
P_tu(x)=\int_X K_t(x,y)u(y)\,d\mu(y).
\]
In this case, $\mathcal L=L_{\mathcal E}$.

From now on, we write
\[
L:=L_{\mathcal E}=\mathcal L,
\qquad
\Delta=-L,
\qquad
P_t=e^{-tL}=e^{t\Delta}.
\]

\begin{remark}
All identities and inequalities involving a merely measurable heat kernel are understood
almost everywhere in the spatial variables, for each fixed choice of the positive time
parameters. Pointwise formulations will be used when a suitable version of the kernel
is available, as in the Riemannian setting.
\end{remark}

\begin{remark}
The assumptions above do not imply that $P_t$ preserves $C(X)\cap L^\infty(X,\mu)$,
or that $P_tu\to u$ uniformly as $t\downarrow0$ for every bounded continuous function $u$.
Any such continuity property must be justified by additional assumptions.
\end{remark}

\begin{remark}
A finite-valued, nondecreasing convex function $f:[0,\infty)\to[0,\infty)$
is locally Lipschitz, including at the endpoint $0$.
However, these assumptions do not imply $f(0)=0$.

For a bounded nonnegative function $u$, convexity and the sub-Markovian property give
\[
f(P_tu)\le P_t(f(u))+(1-P_t1)f(0).
\]
Consequently,
\[
f(P_tu)\le P_t(f(u))
\]
if the semigroup is conservative or if $f(0)=0$.
\end{remark}

\subsection{Volume growth and heat kernel bounds}

The volume growth conditions and heat kernel estimates introduced below
are additional assumptions. Each of them is imposed only in the results
where it is explicitly invoked. 

For positive functions $a,b$, we write $a\asymp b$
when there exist constants $c,C>0$ such that
$cb\le a\le Cb$ throughout the specified range.
We denote the volume of a ball by
\begin{equation}\label{volumepalla}
V(x,r):=\mu(B_r(x)).
\end{equation}

\begin{definition}[Volume doubling condition]\label{def:VD}
The measure $\mu$ satisfies the \emph{volume doubling condition} if there exists
a constant $C_D\ge1$ such that
\begin{equation}\label{eq04}\tag{VD}
V(x,2r)\le C_DV(x,r)
\qquad \text{for every }x\in X,\ r>0.
\end{equation}
\end{definition}

\begin{remark}
Under Definition \ref{MMS}, the positivity and finiteness of $V(x,r)$ follow from
full support, the Radon property, and the relative compactness of balls.
Condition \eqref{eq04} additionally requires a doubling constant independent
of both $x$ and $r$. Local versions restrict the range of radii, while non-uniform
versions may allow the constants to depend on the center.
\end{remark}

\begin{remark}
Condition \eqref{eq04} is equivalent to the existence of constants $M\ge1$ and $s>0$
such that
\begin{equation}\label{eq05bis}
V(x,R)\le M\left(\frac{d(x,y)+R}{r}\right)^s V(y,r)
\end{equation}
for all $x,y\in X$ and $0<r\le R$. It is also equivalent to
\begin{equation}\label{eq05}
V(x,R)\le M\left(\frac{R}{r}\right)^s V(x,r)
\end{equation}
uniformly in $x\in X$ and $0<r\le R$.

Indeed, starting from \eqref{eq04}, and increasing $C_D$ if necessary so that $C_D>1$,
one may take
\[
s=\log_2C_D,
\qquad M=C_D.
\]
The estimate with different centers follows from the same-center
estimate and the inclusion
\[
B_R(x)\subseteq B_{R+d(x,y)}(y).
\]
Conversely, \eqref{eq05} implies \eqref{eq04} with doubling constant $M2^s$.
These constants need not coincide with the original doubling constant.
See, for instance, \cite{GrigHH}.
\end{remark}

\begin{remark}
On a smooth Riemannian manifold, local volume doubling holds for sufficiently small
radii around each point.

If the manifold is complete and its Ricci curvature is bounded from below, then,
for every fixed $R_0>0$, there exists a constant $C_D(R_0)$ such that
\[
V(x,2r)\le C_D(R_0)V(x,r)
\]
for every $x$ and every $0<r\le R_0$. The constant may depend on the dimension,
the curvature bound, and $R_0$.

If the Ricci curvature is nonnegative, the doubling condition holds uniformly
for all centers and all radii; see \cite{SC}.
\end{remark}

\begin{definition}\label{def:volumegrowth}
The volume growth is said to be:
\begin{enumerate}[label=(\roman*)]
\item \emph{uniformly at most polynomial} if there exist $C\ge1$ and $n\in\mathbb N$,
$n\ge1$, such that
\[
V(x,r)\le Cr^nV(x,1)
\qquad \text{for every }x\in X,\ r\ge1;
\]

\item \emph{uniformly at least exponential} if there exist constants $C,\varepsilon>0$
and $r_0>0$ such that
\[
V(x,r)\ge Ce^{\varepsilon r}
\qquad \text{for every }x\in X,\ r\ge r_0.
\]
\end{enumerate}
\end{definition}
The restriction to sufficiently large radii in part (ii) is essential for non-atomic
measures. These uniform definitions are stronger than corresponding growth conditions
imposed at a single fixed center.

\begin{definition}\label{def:HK}
The Dirichlet space $(X,d,\mu,\mathcal E,\mathcal F)$ satisfies the
\emph{heat kernel estimate $(HK_\beta)$}, with $\beta\ge2$, if its semigroup admits
a heat kernel and there exist constants $c_1,c_2,C_1,C_2>0$ such that
\begin{equation}\label{eq11}
K_t(x,y)\ge
\frac{c_1}{V(x,t^{1/\beta})}g_\beta(c_2d(x,y),t)
\end{equation}
and
\begin{equation}\label{eq12}
K_t(x,y)\le
\frac{C_1}{V(x,t^{1/\beta})}g_\beta(C_2d(x,y),t)
\end{equation}
for every $t>0$ and almost every $(x,y)\in X\times X$, where
\begin{equation}\label{eq13}
g_\beta(r,t):=
\exp\left(-\left(\frac{r^\beta}{t}\right)^{1/(\beta-1)}\right).
\end{equation}
\end{definition}

\begin{theorem}\label{thm:equivalence}
Let $(X,d,\mu)$ be as in Definition \ref{MMS}. If $(HK_\beta)$ holds, then $\mu$
satisfies the volume doubling condition \eqref{eq04}.
\end{theorem}

\begin{proof}
Only the lower bound in $(HK_\beta)$ and the sub-Markovian property
are needed.

Fix $x\in X$ and $r>0$, and set
\[
t=(r/2)^\beta.
\]
By Fubini's theorem, for $\mu$-almost every $z\in X$, the lower heat
kernel bound holds for $\mu$-almost every $y\in X$, and
\[
\int_X K_t(z,y)\,d\mu(y)\le1.
\]
Since $\mu$ has full support, we may choose such a point
$z\in B_{r/4}(x)$.
For $y\in B_{2r}(x)$, we have $d(z,y)<3r$, while
\[
B_{r/2}(z)\subseteq B_r(x).
\]
Consequently, for almost every $y\in B_{2r}(x)$,
\[
K_t(z,y)
\ge
\frac{c_1}{V(z,r/2)}
\exp\left(-(6c_2)^{\beta/(\beta-1)}\right)
\ge
\frac{a}{V(x,r)},
\]
where
\[
a:=c_1\exp\left(-(6c_2)^{\beta/(\beta-1)}\right)>0.
\]
Integrating over $B_{2r}(x)$ gives
\[
1\ge \int_{B_{2r}(x)}K_t(z,y)\,d\mu(y)
\ge a\,\frac{V(x,2r)}{V(x,r)}.
\]
Thus the volume doubling condition holds, with
$C_D=\max\{1,a^{-1}\}$.
\end{proof}

\begin{remark}
The converse does not hold in general, even when $\beta=2$.
For the Laplace--Beltrami operator on a complete Riemannian manifold, two-sided Gaussian
heat kernel estimates are characterized by volume doubling together with a uniform
scale-invariant $L^2$-Poincar\'e inequality; see \cite{SC}.

For general Dirichlet forms, analogous characterizations require suitable assumptions
on the form and on the compatibility between the form and the metric.
For $\beta>2$, additional analytic conditions must likewise be stated precisely.
\end{remark}


Further background on heat kernel bounds with power
prefactors, Ahlfors regularity, and the distinction
between local and non-local forms is collected in
Appendix \ref{app:kernel-background}.

We now recall the Gaussian estimates used in the
Riemannian applications; see \cite{LiYau,SC}.

\begin{theorem}\label{thm:ricci}
Let $M$ be a complete connected Riemannian manifold of dimension $n$,
with Riemannian volume $\mu$, and assume that
\[
\operatorname{Ric}\ge0.
\]
Let $K_t$ be the heat kernel of the Laplace--Beltrami operator.
Then there exist positive constants $c_1,c_2,c_3,c_4$, depending only on $n$, such that
\begin{equation}\label{riccipositive}
\frac{c_1}{V(x,\sqrt t)}
\exp\left(-\frac{d(x,y)^2}{c_2t}\right)
\le K_t(x,y)\le
\frac{c_3}{V(x,\sqrt t)}
\exp\left(-\frac{d(x,y)^2}{c_4t}\right)
\end{equation}
for every $x,y\in M$ and every $t>0$.
\end{theorem}

\begin{remark}
If, in addition,
\[
a r^n\le V(x,r)\le b r^n
\qquad \text{for every }x\in M,\ r>0
\]
for some $a,b>0$, then \eqref{riccipositive} yields Gaussian bounds with prefactor
$t^{-n/2}$. The constants then also depend on the volume bounds.

Nonnegative Ricci curvature alone does not imply this additional volume condition.
\end{remark}

\subsection{Exponential heat kernel bounds}

We also consider heat kernels with exponential decay in time.

\begin{definition}\label{def03}
A heat kernel satisfies \eqref{E1} if there exists $\lambda>0$ such that,
for every $t_0>0$, there is a constant $C=C(t_0)>0$ satisfying
\begin{equation}\label{E1}\tag{E1}
K_t(x,y)\le Ce^{-\lambda t}
\end{equation}
for every $t\ge t_0$ and almost every $(x,y)\in X\times X$.

A heat kernel satisfies \eqref{E2} if there exists $\lambda>0$ such that,
for every $\varepsilon\in(0,\lambda)$ and every $t_0>0$, there are
a constant $c=c(\varepsilon,t_0)>0$ and a nonempty relatively compact
open set $\Omega=\Omega(\varepsilon,t_0)\subset X$ such that
\begin{equation}\label{E2}\tag{E2}
K_t(x,y)\ge ce^{-(\lambda+\varepsilon)t}
\end{equation}
for every $t\ge t_0$ and almost every $(x,y)\in\Omega\times\Omega$.

Whenever \eqref{E1} and \eqref{E2} are imposed simultaneously,
the same value of $\lambda$ is used.
\end{definition}

Since $\mu$ has full support and $\Omega$ is nonempty, open, and relatively compact,
\[
0<\mu(\Omega)<\infty.
\]

Let $M$ be a complete connected noncompact Riemannian manifold, and denote by
$\lambda_1(M)$ the bottom of the $L^2$-spectrum of $-\Delta$. For its heat kernel,
\[
\lim_{t\to\infty}\frac{\log K_t(x,y)}{t}=-\lambda_1(M)
\]
locally uniformly in $M\times M$ (see \cite{ChavelKarp}). Consequently, if $\lambda_1(M)>0$,
condition \eqref{E2} holds with $\lambda=\lambda_1(M)$.
In fact, it holds on every prescribed nonempty relatively compact open set $\Omega$,
with a constant depending on $\Omega$, $\varepsilon$, and $t_0$.

Indeed, the locally uniform large-time asymptotics yield the lower
bound for all sufficiently large times. Positivity and continuity
of the heat kernel on
\[
\overline{\Omega}\times\overline{\Omega}\times[t_0,T]
\]
then allow the constant to be decreased so that the same bound holds
for every $t\ge t_0$.

The upper bound \eqref{E1}, which is uniform over the entire manifold, requires
additional analytic control. A sufficient assumption is the Faber--Krahn inequality
\[
\lambda_1(\Omega)\ge \frac{c}{\mu(\Omega)^{2/n}}
\]
for every nonempty relatively compact open set $\Omega\subset M$, where $n=\dim M$
and $\lambda_1(\Omega)$ is the bottom of the Dirichlet spectrum on $\Omega$.
Under this assumption and $\lambda_1(M)>0$, condition \eqref{E1} holds with
$\lambda=\lambda_1(M)$. See \cite{GrilloMeglioliPunzo2023}.

\subsection{Graphs: a separate discussion}

We distinguish discrete weighted graphs from metric graphs.
Discrete graphs do not belong to the class in Definition \ref{MMS};
nevertheless, the general theory of regular Dirichlet forms applies to them.

\subsubsection{Discrete weighted graphs}

Let $G=(V,E)$ be a countable connected graph with at least two vertices. Let
\[
b:V\times V\to[0,\infty)
\]
be symmetric edge weights satisfying
\[
b(x,x)=0,
\qquad
\sum_{y\in V}b(x,y)<\infty
\quad \text{for every }x\in V.
\]
We understand that $x$ and $y$ are adjacent precisely when $b(x,y)>0$.
Let $m:V\to(0,\infty)$, and equip $V$ with the discrete topology and the measure
\[
m(A)=\sum_{x\in A}m(x).
\]

The Hilbert space $\ell^2(V,m)$ has inner product
\[
\langle u,v\rangle_{\ell^2(V,m)}
=\sum_{x\in V}u(x)v(x)m(x).
\]

On finitely supported functions, define
\[
\mathcal E(u,v)
=\frac12\sum_{x,y\in V}b(x,y)(u(x)-u(y))(v(x)-v(y)).
\]
The summability assumption on $b$ ensures that this expression is finite for finitely
supported functions. The form is closable in $\ell^2(V,m)$, and its closure
$(\mathcal E,\mathcal F)$ is a regular Dirichlet form.

Here $\mathcal F$ is specifically the closure of the finitely supported
functions in the form norm; it need not coincide with the space of all
$\ell^2(V,m)$ functions having finite energy.

This form is non-local. Indeed, if $x\ne y$ and $b(x,y)>0$, then
\[
\mathcal E(\mathbf 1_{\{x\}},\mathbf 1_{\{y\}})=-b(x,y)\ne0,
\]
although the two functions have disjoint compact supports.
In particular, the form is not strongly local.
It is a pure jump form, with no killing term.

The associated nonnegative self-adjoint operator $L=L_{\mathcal E}$ satisfies
\[
D(L)\subseteq\mathcal F,
\qquad
\mathcal E(u,v)=\langle Lu,v\rangle_{\ell^2(V,m)}
\]
for every $u\in D(L)$ and $v\in\mathcal F$. For $u\in D(L)$,
\[
(Lu)(x)=\frac{1}{m(x)}\sum_{y\in V}b(x,y)(u(x)-u(y)).
\]
Thus, with $\Delta=-L$,
\[
(\Delta u)(x)=\frac{1}{m(x)}\sum_{y\in V}b(x,y)(u(y)-u(x)).
\]

The semigroup $P_t=e^{-tL}$ admits a heat kernel with respect to the
measure $m$, given by
\[
K_t(x,y):=\frac{(P_t\mathbf 1_{\{y\}})(x)}{m(y)}.
\]
Thus,
\[
P_tu(x)=\sum_{y\in V}K_t(x,y)u(y)m(y).
\]
Since every vertex has positive measure, almost-everywhere identities
in the spatial variables are pointwise identities in this setting.

As for every closed nonnegative symmetric form, $D(L)$ is dense in $\mathcal F$
with respect to the form norm. Moreover, if $u\in\mathcal F$ and $w\in\ell^2(V,m)$
satisfy
\[
\mathcal E(u,v)=\langle w,v\rangle_{\ell^2(V,m)}
\]
for every $v$ in a form core, then
\[
u\in D(L),
\qquad Lu=w.
\]

\begin{remark}
Combinatorial connectedness does not imply connectedness of the discrete topology.
A discrete space with at least two points is not a geodesic metric space in the sense
of Definition \ref{MMS}.

Furthermore, estimates of the form \eqref{Htheta}, with $\alpha>0$ and $\Phi_1(1)>0$,
cannot hold on such a graph for every $t>0$. Indeed, the total mass inequality implies
\[
K_t(x,x)m(x)\le1,
\]
whereas the lower bound in \eqref{Htheta} would give
\[
K_t(x,x)\ge t^{-\alpha/\beta}\Phi_1(0)\longrightarrow\infty
\qquad \text{as }t\downarrow0.
\]
Heat kernel estimates on discrete graphs therefore require a formulation adapted
to their small-scale and short-time behavior.
\end{remark}

\subsubsection{Metric graphs}

For a precise metric graph example, let $X$ be the geometric realization
of a connected, locally finite graph with countably many edges and
at least one edge. Identify each edge $e$ with an interval
$(0,\ell_e)$, and assume that
\[
0<\ell_*\le\ell_e<\infty
\qquad \text{for every edge }e
\]
for some fixed $\ell_*>0$. Equip $X$ with the path metric and the measure
$\mu$ given by Lebesgue measure on each edge, with vertices having
measure zero.

These assumptions ensure that $X$ is a complete, locally compact geodesic metric space.
Define
\[
\begin{split}
\mathcal F=\bigg\{
u\in L^2(X,\mu):\;&
u|_e\in W^{1,2}(0,\ell_e)\text{ for every }e,\\
&
u\text{ has matching traces at each vertex},\\
&
\sum_e\int_0^{\ell_e}|u'_e(s)|^2\,ds<\infty
\bigg\}.
\end{split}
\]
The matching trace condition gives a continuous representative on the metric graph. Set
\[
\mathcal E(u,v)=\sum_e\int_0^{\ell_e}u'_e(s)v'_e(s)\,ds,
\qquad u,v\in\mathcal F.
\]
This is a regular strongly local Dirichlet form.

Its associated nonnegative self-adjoint operator acts on each edge as
\[
(Lu)|_e=-u''_e.
\]
Its domain consists of the functions $u\in\mathcal F$ such that
\[
u_e\in W^{2,2}(0,\ell_e)
\quad\text{for every }e,
\qquad
\sum_e\|u''_e\|_{L^2(0,\ell_e)}^2<\infty,
\]
and the Kirchhoff condition
\[
\sum_{e\sim v}\partial_{\nu_e}u(v)=0
\]
holds at every vertex $v$. Here $\partial_{\nu_e}u(v)$ denotes the derivative at $v$
directed from the vertex into the incident edge, and each incidence is counted separately.
At a vertex of degree one, this condition is Neumann.

Other vertex conditions require a corresponding choice of the form domain and,
in some cases, additional terms in the form.

\section{Statements of the main results}\label{statements}
\setcounter{equation}{0}

We work with the regular symmetric Dirichlet form, the associated
Laplacian $\Delta=-L$, and the heat semigroup $(P_t)_{t\ge0}$ introduced
in Section \ref{mf}. Throughout this section, we assume that the
semigroup admits a nonnegative measurable heat kernel $K_t(x,y)$
in the sense of Definition \ref{def:heatkernel}.
In particular, the semigroup is positivity preserving and sub-Markovian;
conservativity is not assumed.

We always assume
\begin{equation}\label{mainassumptions}\tag{H}
\begin{cases}
\text{(i) $(X,d,\mu)$ is a metric measure space as in Definition \ref{MMS};}\\[3pt]
\text{(ii) $f:[0,\infty)\to[0,\infty)$ is locally Lipschitz, $f(0)=0$, and $f(a)>0$ for $a>0$;}\\[3pt]
\text{(iii) $h\in C([0,\infty))$ and $h(t)>0$ for every $t>0$;}\\[3pt]
\text{(iv) $u_0\in C(X)\cap L^\infty(X,\mu)$, $u_0\ge0$, and $u_0\not\equiv0$.}
\end{cases}
\end{equation}

\begin{definition}\label{defsol}
Let $T\in(0,\infty]$ and let $u_0\in L^\infty(X,\mu)$ be nonnegative.
A nonnegative mild solution of problem \eqref{problema} on $[0,T)$
is a jointly measurable function
\[
u:X\times[0,T)\to[0,\infty)
\]
such that $u(\cdot,0)=u_0$ almost everywhere,
\[
\sup_{0\le t\le\tau}\|u(\cdot,t)\|_{L^\infty(X,\mu)}<\infty
\qquad\text{for every }0<\tau<T,
\]
and, for every $t\in(0,T)$,
\begin{equation}\label{solmild}
\begin{aligned}
u(x,t)
&=(P_tu_0)(x)
+\int_0^t h(s)\bigl(P_{t-s}(f(u(\cdot,s)))\bigr)(x)\,ds\\
&=\int_X K_t(x,y)u_0(y)\,d\mu(y)\\
&\quad+\int_0^t\int_X K_{t-s}(x,y)h(s)f(u(y,s))\,d\mu(y)\,ds
\end{aligned}
\end{equation}
for $\mu$-almost every $x\in X$.
The time integral is understood as a Lebesgue integral of the kernel
expression; its endpoint values are immaterial.

Two mild solutions are identified if their spatial sections agree
$\mu$-almost everywhere for every $t\in[0,T)$.
A mild solution is called global if $T=\infty$.
\end{definition}

For every nonnegative jointly measurable function satisfying the local
boundedness requirement in Definition \ref{defsol}, the integral in
\eqref{solmild} is finite almost everywhere for each time in its interval
of definition. Indeed, if
\[
B_\tau:=\sup_{0\le s\le\tau}\|u(\cdot,s)\|_{L^\infty(X,\mu)},
\]
then positivity and $L^\infty$-contractivity give
\[
\left\|\int_0^t h(s)P_{t-s}(f(u(\cdot,s)))\,ds\right\|_{L^\infty(X,\mu)}
\le \max_{0\le a\le B_\tau}f(a)\int_0^t h(s)\,ds,
\qquad 0<t\le\tau.
\]

When referring to a mild solution on a closed interval $[0,T]$, we
require \eqref{solmild} also at $t=T$ and
\[
\sup_{0\le t\le T}\|u(\cdot,t)\|_{L^\infty(X,\mu)}<\infty.
\]

\begin{proposition}\label{prop:continuation}
In the semigroup framework fixed above, assume that
$f:[0,\infty)\to[0,\infty)$ is locally Lipschitz,
$h\in C([0,\infty))$ is nonnegative, and
$u_0\in L^\infty(X,\mu)$ is nonnegative.
Then there exists a unique maximal nonnegative mild solution on
$[0,T_{\max})$, where $T_{\max}\in(0,\infty]$.
Here maximal means that the solution admits no extension as a mild
solution to a strictly larger interval of the same form.
If $T_{\max}<\infty$, then
\[
\lim_{t\uparrow T_{\max}}\|u(\cdot,t)\|_{L^\infty(X,\mu)}=\infty.
\]
\end{proposition}

\subsection{Blow-up of solutions}

Throughout this subsection, in addition to \eqref{mainassumptions},
we assume that $f$ is nondecreasing and convex and that
\begin{equation}\label{Hf}\tag{H1}
\int_1^\infty\frac{ds}{f(s)}<\infty.
\end{equation}

For $a>0$, define
\begin{equation}\label{J}
J(a):=\int_a^\infty\frac{dz}{f(z)},
\end{equation}
and set $J(0):=\infty$. We also define
\begin{equation}\label{eq44}
H(t):=\int_0^t h(s)\,ds,
\qquad t\ge0.
\end{equation}

The positivity of $f$ on $(0,\infty)$ and \eqref{Hf} ensure that
$J(a)<\infty$ for every $a>0$.
Moreover, $J$ is continuously differentiable and strictly decreasing
on $(0,\infty)$, with
\[
J'(a)=-\frac1{f(a)},
\qquad \lim_{a\to\infty}J(a)=0.
\]
Since $f$ is locally Lipschitz and $f(0)=0$, there exists $L_f>0$ such that
\[
f(a)\le L_fa\qquad\text{for }0\le a\le1.
\]
Consequently,
\[
\lim_{a\downarrow0}J(a)=\infty.
\]
Thus $J$ maps $(0,\infty)$ bijectively onto $(0,\infty)$, and its inverse
\[
J^{-1}:(0,\infty)\to(0,\infty)
\]
is well defined and strictly decreasing.

Our main criterion is expressed in terms of the linear evolution
$P_Tu_0$ at a single time.

\begin{theorem}\label{teo1}
Assume \eqref{mainassumptions} and \eqref{Hf}, and let $f$ be
nondecreasing and convex.
If there exists $T>0$ such that
\begin{equation}\label{generalBU}
\|P_Tu_0\|_{L^\infty(X,\mu)}\ge J^{-1}(H(T)),
\end{equation}
then the maximal nonnegative mild solution of problem \eqref{problema}
satisfies $T_{\max}\le T$ and blows up in the sense of
Proposition \ref{prop:continuation}.
In particular, no global nonnegative mild solution exists in the
sense of Definition \ref{defsol}.
Equivalently, condition \eqref{generalBU} can be written as
\begin{equation}\label{generalBU-J}
H(T)\ge J\bigl(\|P_Tu_0\|_{L^\infty(X,\mu)}\bigr).
\end{equation}
\end{theorem}

\begin{remark}\label{rem:large-data-BU}
Theorem \ref{teo1} includes blow-up for large initial data without
requiring a condition on the long-time behaviour of the heat kernel.
Indeed, let $\varphi\in C(X)\cap L^\infty(X,\mu)$ be nonnegative and
suppose that, for some $T>0$,
\[
\|P_T\varphi\|_{L^\infty(X,\mu)}>0.
\]
For $u_0=A\varphi$, linearity gives
\[
\|P_Tu_0\|_{L^\infty(X,\mu)}=A\|P_T\varphi\|_{L^\infty(X,\mu)}.
\]
Thus the solution blows up no later than $T$ whenever
\[
A\ge\frac{J^{-1}(H(T))}{\|P_T\varphi\|_{L^\infty(X,\mu)}}.
\]
In particular, when $h\equiv1$ and $f(a)=a^q$, with $q>1$, the criterion reads
\[
\|P_Tu_0\|_{L^\infty(X,\mu)}\ge ((q-1)T)^{-1/(q-1)}.
\]
Hence it applies to sufficiently large multiples of $\varphi$ for every
$q>1$, provided $\|P_T\varphi\|_{L^\infty(X,\mu)}>0$.
\end{remark}

The next corollary gives sufficient conditions expressed directly in
terms of the heat kernel.
Choose $x_0\in X$, $r>0$, and $\delta>0$ such that
\[
u_0(x)\ge\delta\qquad\text{for every }x\in B_r(x_0).
\]
Such a choice is possible because $u_0$ is continuous, nonnegative,
and nontrivial. Recall the notation $V(x_0,r)$ from \eqref{volumepalla};
in particular, $0<V(x_0,r)<\infty$.
Define
\begin{equation}\label{infKernel}
\underline{\mathcal K}(t):=
\operatorname*{ess\,inf}_{(x,y)\in B_r(x_0)\times B_r(x_0)}K_t(x,y),
\qquad t>0,
\end{equation}
where the essential infimum is taken with respect to $\mu\otimes\mu$.
Strict positivity of $\underline{\mathcal K}(t)$ is not assumed.

\begin{corollary}\label{corteo1}
Assume the hypotheses of Theorem \ref{teo1}, except for \eqref{generalBU}.
Let $B_r(x_0)$, $\delta$, and $\underline{\mathcal K}$ be chosen as above.
Suppose that at least one of the following conditions holds:
\begin{enumerate}[label=(\roman*)]
\item there exists $T>0$ such that
\begin{equation}\label{finiteBU}
\delta V(x_0,r)\underline{\mathcal K}(T)\ge J^{-1}(H(T));
\end{equation}
\item
\begin{equation}\label{hpBU}
\limsup_{T\to\infty}
\frac{\underline{\mathcal K}(T)}{J^{-1}(H(T))}=\infty.
\end{equation}
\end{enumerate}
Then the maximal nonnegative mild solution of \eqref{problema}
blows up in finite time. Under condition (i), $T_{\max}\le T$.
\end{corollary}

\begin{remark}\label{rem:comparison-LSZ}
We compare Theorem \ref{teo1} with the nonexistence criterion of
Lenz, Schmidt and Zimmermann
\cite[Theorem 4.1]{LSZ}.
In the autonomous case $h\equiv1$, their criterion requires the
existence of $T>0$ and a measurable set $D\subseteq X$, with
$0<\mu(D)<\infty$, such that
\[
\frac1{\mu(D)}\int_D(P_Tu_0)(x)\,d\mu(x)\ge J^{-1}(T).
\]
For bounded initial data, this condition implies \eqref{generalBU}, since
\[
\frac1{\mu(D)}\int_D(P_Tu_0)(x)\,d\mu(x)
\le\|P_Tu_0\|_{L^\infty(X,\mu)}.
\]
Conversely, a strict inequality
$\|P_Tu_0\|_{L^\infty(X,\mu)}>J^{-1}(T)$ implies the existence of
such a set $D$, whereas this need not follow from equality of the norm
with the threshold at that same time.

Their result is formulated for $L^p$-valued mild solutions,
$1\le p<\infty$, and strongly continuous sub-Markovian semigroups,
without requiring an integral kernel.
Our formulation uses a heat kernel and solutions bounded on compact
time intervals, without requiring the initial datum to belong to any
$L^p(X,\mu)$ with $p<\infty$.
It also allows a time-dependent coefficient $h$, replacing the
threshold $J^{-1}(T)$ by $J^{-1}(H(T))$.
Both criteria retain the dependence on the initial datum and do not
require asymptotic heat kernel assumptions.
\end{remark}

\subsection{Global existence}

Concerning global-in-time existence, we establish a criterion based on
a weighted heat kernel bound and a contraction argument.
For this result, neither convexity nor monotonicity of the nonlinearity
is required. Throughout this subsection, we assume
\eqref{mainassumptions} and retain the semigroup and heat kernel
framework fixed at the beginning of this section.
No separate growth assumption on $f$ at infinity is imposed beyond
the integral condition stated below.

Define
\begin{equation}\label{supKernel}
\overline{\mathcal K}(t)
:=
\operatorname*{ess\,sup}_{(x,y)\in X\times X}K_t(x,y),
\qquad t>0,
\end{equation}
where the essential supremum is taken with respect to $\mu\otimes\mu$.
We assume that $\overline{\mathcal K}$ is finite-valued and measurable
on $(0,\infty)$.

Fix $\gamma>0$ and a point $y_0\in X$ for which the chosen
heat kernel satisfies
\begin{equation}\label{columnsemigroup}
\int_X K_{t-s}(x,z)K_{s+\gamma}(z,y_0)\,d\mu(z)
=
K_{t+\gamma}(x,y_0)
\end{equation}
for every $0\le s<t$ and $\mu$-almost every $x\in X$.
We also require that
\begin{equation}\label{columnbound}
\|K_{t+\gamma}(\cdot,y_0)\|_{L^\infty(X,\mu)}
\le \overline{\mathcal K}(t+\gamma)
\qquad\text{for every }t\ge0.
\end{equation}
Spatial identities and inequalities are understood according to
the almost-everywhere convention fixed in Section \ref{heats}.

Let $\omega:[0,\infty)\to[1,\infty)$ satisfy
\begin{equation}\label{omega}
\omega\in C([0,\infty)),
\qquad
\omega(t)\ge1,
\qquad
\omega\text{ is nondecreasing}.
\end{equation}
For $M>0$, set
\[
R_M(t):=M\omega(t)\overline{\mathcal K}(t+\gamma),
\qquad t\ge0.
\]

Define
\begin{equation}\label{def:g}
g(\xi):=\frac{f(\xi)}{\xi},
\qquad \xi>0.
\end{equation}
Since $f$ is locally Lipschitz and $f(0)=0$, the function $g$
is nonnegative and continuous on $(0,\infty)$, and is bounded
on $(0,r]$ for every $r>0$.
No extension of $g$ to the origin is needed.

For $t\ge0$, define
\begin{equation}\label{thetaM}
\theta_M(t):=
\begin{cases}
\displaystyle\sup_{0<\xi\le R_M(t)}g(\xi),
& R_M(t)>0,\\[6pt]
0, & R_M(t)=0,
\end{cases}
\end{equation}
and
\begin{equation}\label{sigmaM}
\sigma_M(t):=
\begin{cases}
\displaystyle\sup_{0\le a<b\le R_M(t)}
\frac{|f(b)-f(a)|}{b-a},
& R_M(t)>0,\\[8pt]
0, & R_M(t)=0.
\end{cases}
\end{equation}
Thus, $\sigma_M(t)$ is the optimal Lipschitz constant of $f$
on $[0,R_M(t)]$, with the convention above when this interval
reduces to a single point.

Both functions are finite-valued and measurable.
Indeed, the corresponding suprema are nondecreasing functions
of the radius and are therefore Borel measurable.
Since $f(0)=0$,
\[
0\le\theta_M(t)\le\sigma_M(t)
\qquad\text{for every }t\ge0.
\]
Their dependence on $\gamma$ and $\omega$ is omitted from
the notation.

If $f\in C^1([0,\infty))$ and $R_M(t)>0$, then
\[
\sigma_M(t)
=
\max_{0\le\xi\le R_M(t)}|f'(\xi)|.
\]

For some $\delta\in(0,1)$, we consider the conditions
\begin{equation}\label{eq53}
\int_0^t \theta_M(s)h(s)\omega(s)\,ds
\le\delta\omega(t),
\qquad t>0,
\end{equation}
and
\begin{equation}\label{LM}
\int_0^t \sigma_M(s)h(s)\omega(s)\,ds
\le\delta\omega(t),
\qquad t>0.
\end{equation}
Condition \eqref{LM} implies \eqref{eq53}.
We retain both formulations to distinguish the estimates
used for invariance and contraction.

\begin{theorem}\label{teo2}
Assume \eqref{mainassumptions}, and let the heat semigroup
and its kernel be as fixed above.
Suppose that $\overline{\mathcal K}$, defined in
\eqref{supKernel}, is finite-valued and measurable on $(0,\infty)$.

Assume that there exist $\gamma>0$, $y_0\in X$, $M>0$,
$\delta\in(0,1)$, and a function $\omega$ satisfying
\eqref{omega}, such that \eqref{columnsemigroup},
\eqref{columnbound}, and \eqref{LM} hold.
Suppose also that
\begin{equation}\label{eq55}
0\le u_0(x)\le
\varepsilon K_\gamma(x,y_0)\omega(0)
\qquad
\text{for $\mu$-almost every }x\in X,
\end{equation}
where
\begin{equation}\label{eq54}
0<\varepsilon\le(1-\delta)M.
\end{equation}
Then problem \eqref{problema} admits a unique global
nonnegative mild solution in the sense of Definition \ref{defsol}.
Moreover,
\begin{equation}\label{globalbound}
0\le u(x,t)\le M K_{t+\gamma}(x,y_0)\omega(t)
\end{equation}
for every $t\ge0$ and $\mu$-almost every $x\in X$.
\end{theorem}

\section{Blow-up of solutions: proofs}\label{blowup-proof}
\setcounter{equation}{0}

We first prove the local existence and continuation result stated in
Proposition \ref{prop:continuation}.

\begin{proof}[Proof of Proposition \ref{prop:continuation}]
The semigroup is positivity preserving and contractive on
$L^\infty(X,\mu)$.
Fix $M>\|u_0\|_{L^\infty(X,\mu)}$.
For sufficiently small $\eta>0$, the right-hand side of \eqref{solmild}
defines a contraction on the closed subset
\[
\{v\in L^\infty(X\times(0,\eta),\mu\otimes dt):0\le v\le M
\text{ almost everywhere}\},
\]
equipped with the $L^\infty(X\times(0,\eta),\mu\otimes dt)$ distance.
The integral is interpreted through the kernel, as in Definition
\ref{defsol}, and is independent of the chosen representative in this
space by Tonelli's theorem.

Indeed, invariance follows from
\[
\|u_0\|_{L^\infty(X,\mu)}
+\eta\sup_{0\le s\le\eta}h(s)\max_{0\le a\le M}f(a)\le M,
\]
while the contraction property follows by requiring
\[
\eta\sup_{0\le s\le\eta}h(s)L_M<1,
\]
where $L_M$ is a Lipschitz constant of $f$ on $[0,M]$.

Let $v$ denote the fixed point, and choose a jointly measurable
representative with values in $[0,M]$.
For every $t\in(0,\eta]$, define $u(\cdot,t)$ by the right-hand side
of \eqref{solmild}, with $v$ in the nonlinear term, and set $u(\cdot,0)=u_0$.
On the measurable set where this expression is not finite or lies
outside $[0,M]$, replace its value by $0$; each spatial section of
this set is null by the preceding estimates.
This gives a jointly measurable representative with
\[
\sup_{0\le t\le\eta}\|u(\cdot,t)\|_{L^\infty(X,\mu)}\le M.
\]
Moreover, $u=v$ almost everywhere on $X\times(0,\eta)$.
Tonelli's theorem therefore shows that replacing $v$ by $u$ in the
nonlinear integral leaves its value unchanged, for every fixed $t$,
for almost every $x$.
Thus $u$ satisfies \eqref{solmild} for every $t\in(0,\eta]$ in the
sense of Definition \ref{defsol}.

Uniqueness among locally bounded mild solutions follows from the same
estimates. Let $u_1,u_2$ be mild solutions on $[0,\tau]$, choose
$B\ge\sup_{0\le t\le\tau}\max_{i=1,2}\|u_i(t)\|_{L^\infty(X,\mu)}$,
let $L_B$ be a Lipschitz constant of $f$ on $[0,B]$, and set
$\bar h:=\max_{[0,\tau]}h$. Define
\[
t^*:=\sup\bigl\{t\in[0,\tau]:u_1(s)=u_2(s)\text{ a.e. for every }
s\in[0,t]\bigr\}.
\]
For $t^*\le t\le\tau$, the mild formulation, positivity, and
$L^\infty$-contractivity give
\[
\|u_1(t)-u_2(t)\|_{L^\infty(X,\mu)}
\le L_B\bar h\,(t-t^*)
\sup_{t^*\le s\le t}\|u_1(s)-u_2(s)\|_{L^\infty(X,\mu)}.
\]
Taking the supremum over $t\in[t^*,t_1]$ with
$0<t_1-t^*<(L_B\bar h)^{-1}$ shows that $u_1=u_2$ on $[0,t_1]$,
which contradicts the definition of $t^*$ unless $t^*=\tau$.

The semigroup property and the mild formulation imply the restart identity
\[
u(t)=P_{t-t_0}u(t_0)
+\int_{t_0}^t P_{t-s}\bigl(h(s)f(u(s))\bigr)\,ds
\]
whenever $0\le t_0<t$ belong to the interval of existence.
Here and below, identities of spatial sections hold almost everywhere,
and time integrals are interpreted through the kernel.

Conversely, let $u$ be a mild solution on $[0,t_0]$, and let $w$ be a
mild solution on $[t_0,t_1]$ of the problem restarted at $t_0$ with
initial datum $u(t_0)$, that is,
\[
w(t)=P_{t-t_0}u(t_0)
+\int_{t_0}^t P_{t-s}\bigl(h(s)f(w(s))\bigr)\,ds,
\qquad t_0<t\le t_1.
\]
Then the concatenation $\tilde u:=u$ on $[0,t_0]$ and $\tilde u:=w$
on $(t_0,t_1]$ is a mild solution on $[0,t_1]$.
Indeed, for $t_0<t\le t_1$, the semigroup property and Tonelli's
theorem give
\[
P_{t-t_0}u(t_0)
=P_tu_0+\int_0^{t_0}P_{t-s}\bigl(h(s)f(u(s))\bigr)\,ds,
\]
and inserting this identity into the formula for $w$ yields
\eqref{solmild} for $\tilde u$ at time $t$.
Consequently, the local construction can be repeated at any such $t_0$,
with $\|u_0\|_{L^\infty(X,\mu)}$ replaced by
$\|u(t_0)\|_{L^\infty(X,\mu)}$ and $\sup_{0\le s\le\eta}h(s)$ replaced by
$\sup_{t_0\le s\le t_0+\eta}h(s)$ in the invariance and contraction
conditions.
Since mild solutions agree on their common intervals of existence by
uniqueness, their union is the unique maximal solution.

Suppose that $T_{\max}<\infty$ and that there exists a sequence
$t_j\uparrow T_{\max}$ for which $\|u(t_j)\|_{L^\infty(X,\mu)}$
remains bounded.
Choose a common bound for these norms and then choose $M$ larger than
that bound. Since $h$ is bounded on $[0,T_{\max}+1]$, the same
invariance and contraction estimates provide a common positive
existence time $\eta\le1$ for the problems restarted at $t_j$.
By the concatenation property, gluing $u|_{[0,t_j]}$ with the solution
restarted at $t_j$ yields a mild solution on $[0,t_j+\eta]$.
For sufficiently large $j$, we have $t_j+\eta>T_{\max}$, contradicting
maximality. This proves the asserted blow-up alternative.
\end{proof}

We next establish a necessary condition for the existence of a bounded
solution up to a prescribed time $T$.

\begin{lemma}\label{lemma2}
Assume the hypotheses of Theorem \ref{teo1}, except for \eqref{generalBU}.
If $u$ is a nonnegative mild solution on $[0,T]$ satisfying
\[
\sup_{0\le t\le T}\|u(t)\|_{L^\infty(X,\mu)}<\infty,
\]
then
\begin{equation}\label{necessaryJ}
H(T)<J\bigl((P_Tu_0)(x)\bigr)
\end{equation}
for $\mu$-almost every $x\in X$, with $J(0)=\infty$.
Moreover,
\begin{equation}\label{eq47}
\|P_Tu_0\|_{L^\infty(X,\mu)}<J^{-1}(H(T)).
\end{equation}
\end{lemma}

\begin{proof}
For $0\le t<T$, consider
\begin{equation}\label{eq42}
\Phi_x^T(t)\equiv\Phi_x(t):=(P_{T-t}u(t))(x).
\end{equation}
In particular,
\begin{equation}\label{eq43}
\Phi_x(0)=(P_Tu_0)(x)
=\int_X K_T(x,z)u_0(z)\,d\mu(z).
\end{equation}
Initially, these expressions are understood almost everywhere in $x$
for each fixed $t$.
Applying $P_{T-t}$ to \eqref{solmild}, Tonelli's theorem and the
semigroup property give, for every $0<t<T$,
\begin{equation}\label{eq411}
\begin{aligned}
\Phi_x(t)
&=(P_Tu_0)(x)
+\int_0^t h(s)\bigl(P_{T-s}(f(u(s)))\bigr)(x)\,ds\\
&=\Phi_x(0)
+\int_0^t\int_X K_{T-s}(x,y)h(s)f(u(y,s))\,d\mu(y)\,ds
\end{aligned}
\end{equation}
for almost every $x$.

Choose
\[
M>1+\sup_{0\le s\le T}\|u(s)\|_{L^\infty(X,\mu)}.
\]
Positivity and $L^\infty$-contractivity give
$0\le\Phi_x(t)\le M$ for almost every $(x,t)$.
The integrand in \eqref{eq411} is bounded by $h(s)f(M)$ for almost
every $(x,s)$.
By Fubini's theorem, for almost every $x$ the right-hand side of
\eqref{eq411} defines an absolutely continuous function on $[0,T]$.
We henceforth use this representative of $\Phi_x$.
For almost every $x$, it agrees with \eqref{eq42} for almost every $t$;
for each fixed $t<T$, the two spatial sections also agree almost everywhere.
Continuity of the representative extends its bounds to every time:
\[
0\le\Phi_x(t)\le M\qquad\text{for every }0\le t\le T.
\]
Moreover,
\[
\Phi_x'(t)=h(t)\bigl(P_{T-t}(f(u(t)))\bigr)(x)
\]
for almost every $t\in(0,T)$ and almost every $x$.

By the sub-Markovian Jensen inequality recalled in Section \ref{heats},
convexity and $f(0)=0$ imply
\begin{equation}\label{eq410}
\bigl(P_{T-t}(f(u(t)))\bigr)(x)
\ge f\bigl((P_{T-t}u(t))(x)\bigr)=f(\Phi_x(t))
\end{equation}
for almost every $(x,t)$.
Indeed, the measure $K_{T-t}(x,y)\,d\mu(y)$ has mass at most one;
assigning the missing mass to the value $0$ gives Jensen's inequality
with the additional term $(1-P_{T-t}1)(x)f(0)=0$.
It follows that
\begin{equation}\label{eq412}
\Phi_x'(t)\ge h(t)f(\Phi_x(t))
\end{equation}
for almost every $(x,t)$.
By Fubini's theorem, we can choose a single set $X_T\subseteq X$
of full measure on which the absolute continuity and bounds above hold,
and \eqref{eq412} holds for almost every $t\in(0,T)$.
In particular, $\Phi_x$ is nondecreasing for every $x\in X_T$.

Fix $x\in X_T$ and write
\[
a:=\Phi_x(0)=(P_Tu_0)(x).
\]
Assume first that $a>0$.
Since $a\le\Phi_x(t)\le M$ and $J$ is continuously differentiable
on $(0,\infty)$, the function
\begin{equation}\label{eq422b}
F(t):=J(\Phi_x(t))
\end{equation}
is absolutely continuous on $[0,T]$.
Using \eqref{eq412}, we obtain
\begin{equation}\label{eq417}
F'(t)=-\frac{\Phi_x'(t)}{f(\Phi_x(t))}\le-h(t)
\end{equation}
for almost every $t\in(0,T)$.
Integration yields
\[
J(\Phi_x(t))\le J(a)-H(t),\qquad 0\le t<T.
\]
Since $\Phi_x(t)\le M$ and $J$ is decreasing,
\[
J(\Phi_x(t))\ge J(M)>0.
\]
Consequently, letting $t\uparrow T$ gives
\begin{equation}\label{uniform-J-gap}
H(T)+J(M)\le J(a).
\end{equation}
This proves \eqref{necessaryJ} when $a>0$.
When $a=0$, that inequality follows from $J(0)=\infty$.

The constant $M$ is independent of $x$.
By \eqref{uniform-J-gap} and the monotonicity of $J$,
\[
(P_Tu_0)(x)\le J^{-1}\bigl(H(T)+J(M)\bigr)
\]
for almost every $x$ with $(P_Tu_0)(x)>0$.
The same inequality is immediate where $(P_Tu_0)(x)=0$.
Taking the essential supremum gives
\[
\|P_Tu_0\|_{L^\infty(X,\mu)}
\le J^{-1}\bigl(H(T)+J(M)\bigr)<J^{-1}(H(T)),
\]
because $J(M)>0$ and $J^{-1}$ is strictly decreasing.
This proves \eqref{eq47}, without requiring the essential supremum
to be attained.
\end{proof}

\begin{proof}[Proof of Theorem \ref{teo1}]
Let $u$ be the maximal solution provided by Proposition
\ref{prop:continuation}, and let $T>0$ satisfy \eqref{generalBU}.
Suppose that $T_{\max}>T$.
Then $u$ is bounded on $[0,T]$, and Lemma \ref{lemma2} gives
\[
\|P_Tu_0\|_{L^\infty(X,\mu)}<J^{-1}(H(T)),
\]
contradicting \eqref{generalBU}, including its equality case.
Therefore $T_{\max}\le T<\infty$.
The blow-up alternative in Proposition \ref{prop:continuation} yields
\[
\lim_{t\uparrow T_{\max}}\|u(\cdot,t)\|_{L^\infty(X,\mu)}=\infty.
\]
The equivalence between \eqref{generalBU} and \eqref{generalBU-J}
follows from the strict monotonicity of $J$ when
$\|P_Tu_0\|_{L^\infty(X,\mu)}>0$.
When this norm is zero, both conditions fail, since
$J^{-1}(H(T))>0$ and $J(0)=\infty$.
\end{proof}

We finally record the lower bound on the linear evolution used in the
heat kernel criteria.

\begin{lemma}\label{lemma1}
Let $u_0$ satisfy \eqref{mainassumptions}, and let $x_0$, $r$, and
$\delta$ be chosen as in Corollary \ref{corteo1}.
Then, for every $T>0$,
\begin{equation}\label{eq40}
(P_Tu_0)(x)\ge\delta\,\underline{\mathcal K}(T)V(x_0,r)
\end{equation}
for $\mu$-almost every $x\in B_r(x_0)$.
\end{lemma}

\begin{proof}
By the definition of the essential infimum in \eqref{infKernel} and
Fubini's theorem, for almost every $x\in B_r(x_0)$,
\[
\begin{aligned}
(P_Tu_0)(x)
&=\int_X K_T(x,y)u_0(y)\,d\mu(y)\\
&\ge\int_{B_r(x_0)}K_T(x,y)u_0(y)\,d\mu(y)\\
&\ge\delta\int_{B_r(x_0)}K_T(x,y),d\mu(y)\\
&\ge\delta\,\underline{\mathcal K}(T)V(x_0,r).
\end{aligned}
\]
\end{proof}

\begin{proof}[Proof of Corollary \ref{corteo1}]
Since $V(x_0,r)>0$, Lemma \ref{lemma1} gives
\[
\|P_Tu_0\|_{L^\infty(X,\mu)}
\ge\delta V(x_0,r)\underline{\mathcal K}(T).
\]
Thus \eqref{finiteBU} implies \eqref{generalBU}, and Theorem \ref{teo1}
applies.
If \eqref{hpBU} holds, there exists $T>0$ such that
\[
\frac{\underline{\mathcal K}(T)}{J^{-1}(H(T))}
>\frac1{\delta V(x_0,r)}.
\]
Therefore condition (i) holds, and the conclusion follows from the
first part.
\end{proof}

\section{Global existence: proofs}\label{globalex-proof}
\setcounter{equation}{0}

We prove Theorem \ref{teo2} by applying the Banach contraction
principle in a weighted space of measurable functions.

\begin{definition}\label{funcspace}
Fix $M>0$, $\gamma>0$, $y_0\in X$, and a function $\omega$
satisfying \eqref{omega}. Set
\[
q(x,t):=K_{t+\gamma}(x,y_0)\omega(t),
\qquad x\in X,\quad t>0.
\]
We denote by
$\mathcal M(M,\gamma,y_0,\omega)\equiv\mathcal M$
the set of equivalence classes of measurable functions
satisfying
\begin{equation}\label{funcspacedef}
\mathcal M:=
\left\{
u:X\times(0,\infty)\to[0,\infty):
0\le u(x,t)\le Mq(x,t)
\text{ for almost every }(x,t)
\right\}.
\end{equation}
Functions are identified when they agree
$\mu\otimes dt$-almost everywhere.
For $u,v\in\mathcal M$, define
\begin{equation}\label{eq52}
\rho(u,v):=
\operatorname*{ess\,sup}_{(x,t)\in X\times(0,\infty)}
\frac{|u(x,t)-v(x,t)|}{q(x,t)}.
\end{equation}
The quotient is defined to be zero on $\{q=0\}$,
where every member of $\mathcal M$ vanishes almost everywhere.
\end{definition}

The metric space $(\mathcal M,\rho)$ is nonempty and complete.
Indeed, $0\in\mathcal M$, and the map
\[
u\longmapsto
\begin{cases}
u/q, & q>0,\\
0, & q=0,
\end{cases}
\]
identifies $\mathcal M$ isometrically with the closed subset
of $L^\infty(X\times(0,\infty),\mu\otimes dt)$ consisting
of functions taking values in $[0,M]$ and vanishing almost
everywhere on $\{q=0\}$.

Every member of $\mathcal M$ is essentially bounded on
$X\times(0,T)$ for each $T>0$.
In fact, \eqref{columnsemigroup} with $s=0$ and
$L^\infty$-contractivity give
\[
\begin{aligned}
\|K_{t+\gamma}(\cdot,y_0)\|_{L^\infty(X,\mu)}
&=
\|P_tK_\gamma(\cdot,y_0)\|_{L^\infty(X,\mu)}\\
&\le
\|K_\gamma(\cdot,y_0)\|_{L^\infty(X,\mu)}\\
&\le\overline{\mathcal K}(\gamma),
\end{aligned}
\]
for every $t>0$.
Consequently,
\[
\|u\|_{L^\infty(X\times(0,T),\mu\otimes dt)}
\le M\omega(T)\overline{\mathcal K}(\gamma).
\]

For $u\in\mathcal M$, define
\begin{equation}\label{psi}
(\Psi u)(x,t)
:=
\int_X K_t(x,z)u_0(z)\,d\mu(z)
+
\int_0^t\int_X
K_{t-s}(x,z)h(s)f(u(z,s))\,d\mu(z)\,ds,
\qquad t>0.
\end{equation}
The values at the endpoints of the time integral are immaterial.
Joint measurability of the kernel and Tonelli's theorem ensure
measurability of this expression.
Moreover, changing $u$ on a set of product measure zero does
not change $(\Psi u)(\cdot,t)$ almost everywhere for any
$t>0$, so $\Psi$ is well defined on equivalence classes.

We first show that $\Psi$ maps $\mathcal M$ into itself.

\begin{lemma}\label{lemma-psi}
Assume \eqref{columnsemigroup} and \eqref{columnbound}.
Suppose that \eqref{eq55}, \eqref{eq54}, and \eqref{eq53}
hold, with the same parameters and weight as in
Definition \ref{funcspace}. Then
\[
\Psi(\mathcal M)\subseteq\mathcal M.
\]
\end{lemma}

\begin{proof}
Let $u\in\mathcal M$. By \eqref{columnbound},
\[
0\le u(z,s)
\le M\omega(s)K_{s+\gamma}(z,y_0)
\le R_M(s)
\]
for almost every $(z,s)$.
The definition of $\theta_M$, together with $f(0)=0$, gives
\[
0\le f(u(z,s))
\le\theta_M(s)u(z,s)
\le
M\theta_M(s)\omega(s)K_{s+\gamma}(z,y_0).
\]

Using \eqref{eq55} in \eqref{psi}, we obtain, for every $t>0$
and $\mu$-almost every $x\in X$,
\begin{equation}\label{eq56}
\begin{aligned}
0\le(\Psi u)(x,t)
&\le
\varepsilon\omega(0)
\int_X K_t(x,z)K_\gamma(z,y_0)\,d\mu(z)\\
&\quad+
M\int_0^t h(s)\theta_M(s)\omega(s)
\int_X K_{t-s}(x,z)K_{s+\gamma}(z,y_0)\,d\mu(z)\,ds.
\end{aligned}
\end{equation}
For each fixed $t>0$, joint measurability and Fubini's theorem
allow us to apply \eqref{columnsemigroup} inside the time
integral. Therefore,
\[
\begin{aligned}
(\Psi u)(x,t)
&\le
K_{t+\gamma}(x,y_0)
\left[
\varepsilon\omega(0)
+
M\int_0^t h(s)\theta_M(s)\omega(s)\,ds
\right]\\
&\le
K_{t+\gamma}(x,y_0)
\left[
\varepsilon\omega(0)+\delta M\omega(t)
\right]\\
&\le
(\varepsilon+\delta M)
K_{t+\gamma}(x,y_0)\omega(t),
\end{aligned}
\]
where the last inequality follows from the monotonicity of
$\omega$.
Since \eqref{eq54} implies $\varepsilon+\delta M\le M$,
we conclude that
\[
0\le(\Psi u)(x,t)\le Mq(x,t)
\]
for every $t>0$ and $\mu$-almost every $x\in X$.
In particular, $\Psi u\in\mathcal M$.
\end{proof}

We next prove that $\Psi$ is a contraction on $(\mathcal M,\rho)$.

\begin{proposition}\label{prop-contraction}
Assume \eqref{columnsemigroup}, \eqref{columnbound},
\eqref{eq55}, \eqref{eq54}, and \eqref{LM}, with the same
parameters and weight as in Definition \ref{funcspace}.
Then $\Psi$ maps $\mathcal M$ into itself and
\begin{equation}\label{eq57}
\rho(\Psi u,\Psi v)\le\delta\,\rho(u,v)
\qquad
\text{for every }u,v\in\mathcal M.
\end{equation}
\end{proposition}

\begin{proof}
Since $\theta_M\le\sigma_M$, condition \eqref{LM} implies
\eqref{eq53}. Thus Lemma \ref{lemma-psi} gives
$\Psi(\mathcal M)\subseteq\mathcal M$.

Let $u,v\in\mathcal M$.
For every $t>0$ and $\mu$-almost every $x\in X$,
\begin{equation}\label{eq58}
|(\Psi u-\Psi v)(x,t)|
\le
\int_0^t\int_X
K_{t-s}(x,z)h(s)
|f(u(z,s))-f(v(z,s))|\,d\mu(z)\,ds.
\end{equation}
Since $u(z,s),v(z,s)\in[0,R_M(s)]$ almost everywhere,
\[
|f(u(z,s))-f(v(z,s))|
\le\sigma_M(s)|u(z,s)-v(z,s)|.
\]
Moreover, the definition of $\rho$ gives
\[
|u(z,s)-v(z,s)|
\le
\rho(u,v)K_{s+\gamma}(z,y_0)\omega(s)
\]
for almost every $(z,s)$.

Substituting these estimates into \eqref{eq58} and using
\eqref{columnsemigroup}, we obtain
\[
\begin{aligned}
|(\Psi u-\Psi v)(x,t)|
&\le
\rho(u,v)\int_0^t h(s)\sigma_M(s)\omega(s)
\int_X K_{t-s}(x,z)K_{s+\gamma}(z,y_0)\,d\mu(z)\,ds\\
&=
\rho(u,v)K_{t+\gamma}(x,y_0)
\int_0^t h(s)\sigma_M(s)\omega(s)\,ds\\
&\le
\delta\,\rho(u,v)K_{t+\gamma}(x,y_0)\omega(t)\\
&=
\delta\,\rho(u,v)q(x,t).
\end{aligned}
\]
On $\{q>0\}$, divide by $q$ and take the essential supremum.
On $\{q=0\}$, both $\Psi u$ and $\Psi v$ vanish almost
everywhere. Hence
\[
\rho(\Psi u,\Psi v)\le\delta\,\rho(u,v).
\]
\end{proof}

\begin{proof}[Proof of Theorem \ref{teo2}]
By Proposition \ref{prop-contraction}, $\Psi$ is a contraction
from the nonempty complete metric space $(\mathcal M,\rho)$
into itself.
The Banach contraction principle therefore yields a unique
fixed point in $\mathcal M$.

Let $v$ be a nonnegative measurable representative of this
fixed point. For every $t>0$, define
\[
u(\cdot,t):=(\Psi v)(\cdot,t).
\]
On the measurable set where this expression is not finite,
replace its value by $0$.
For each fixed $t>0$, the corresponding spatial section
is null by the estimates in Lemma \ref{lemma-psi}.
Thus $u$ is a finite-valued, nonnegative jointly measurable
function and
\[
u=v
\qquad\text{almost everywhere in }X\times(0,\infty).
\]

Replacing $v$ by $u$ does not alter the integral term in
\eqref{psi}. Therefore, for every $t>0$ and $\mu$-almost every
$x\in X$,
\[
u(x,t)
=
\int_X K_t(x,y)u_0(y)\,d\mu(y)
+
\int_0^t\int_X
K_{t-s}(x,y)h(s)f(u(y,s))\,d\mu(y)\,ds.
\]
Set $u(\cdot,0)=u_0$.

The estimates in Lemma \ref{lemma-psi} show that
\eqref{globalbound} holds for every $t>0$.
At $t=0$, it follows from \eqref{eq55} and
$\varepsilon\le M$.
Moreover, for every $T>0$,
\[
\sup_{0\le t\le T}
\|u(\cdot,t)\|_{L^\infty(X,\mu)}
\le M\omega(T)\overline{\mathcal K}(\gamma)<\infty.
\]
Hence $u$ is a global nonnegative mild solution in the sense
of Definition \ref{defsol}.

Uniqueness in this class follows from Proposition
\ref{prop:continuation}, applied on each finite time interval.
In particular, uniqueness holds among all nonnegative mild
solutions, without requiring them a priori to belong to
the weighted space $\mathcal M$.
\end{proof}

\begin{remark}\label{rem05}
Let $\delta_0>0$, and define
\[
k(t):=\|K_{t+\gamma}(\cdot,y_0)\|_{L^\infty(X,\mu)},
\qquad t\ge0.
\]
Assume, in addition, that $k$ is positive and continuous
on $[0,\infty)$.
The column semigroup identities and
$L^\infty$-contractivity imply
\[
k(t+s)\le k(t)
\qquad\text{for every }s,t\ge0.
\]
Therefore,
\[
\omega(t):=\frac{k(0)}{k(t)},
\qquad t\ge0,
\]
satisfies \eqref{omega}, with $\omega(0)=1$.

Choose $M>0$ such that
\[
Mk(0)\le\delta_0.
\]
Every $u\in\mathcal M$ then satisfies
\[
0\le u(x,t)
\le
M K_{t+\gamma}(x,y_0)\frac{k(0)}{k(t)}
\le Mk(0)\le\delta_0
\]
for almost every $(x,t)$.

For this choice of weight and $M$, the proofs of
Lemma \ref{lemma-psi} and Proposition
\ref{prop-contraction} remain valid with $\theta_M$ and
$\sigma_M$ replaced, respectively, by
\[
\widetilde\theta
:=
\sup_{0<\xi\le\delta_0}\frac{f(\xi)}{\xi},
\qquad
\widetilde\sigma
:=
\sup_{0\le a<b\le\delta_0}
\frac{|f(b)-f(a)|}{b-a}.
\]
Both constants are finite, and
\[
0\le\widetilde\theta\le\widetilde\sigma.
\]
Consequently, in place of \eqref{LM}, it is sufficient
to assume that
\[
\widetilde\sigma
\int_0^t h(s)\omega(s)\,ds
\le\delta\omega(t)
\qquad\text{for every }t>0
\]
for some $\delta\in(0,1)$, together with \eqref{eq55} and
\[
0<\varepsilon\le(1-\delta)M.
\]
Equivalently, the integral condition is
\[
\widetilde\sigma\,k(t)
\int_0^t\frac{h(s)}{k(s)}\,ds
\le\delta
\qquad\text{for every }t>0.
\]

This is a variant obtained by repeating the contraction
argument with the sharper bound $u\le\delta_0$.
The new integral condition need not imply \eqref{LM},
because the radius $R_M(t)$ defined using
$\overline{\mathcal K}(t+\gamma)$ may exceed $\delta_0$.

Under these assumptions, the same construction produces
a global nonnegative mild solution satisfying
\begin{equation}\label{eqdelta0}
0\le u(x,t)\le\delta_0
\end{equation}
for every $t\ge0$ and $\mu$-almost every $x\in X$.
Uniqueness follows again from Proposition
\ref{prop:continuation}.

If $f\in C^1([0,\delta_0])$, then
\[
\widetilde\sigma
=
\max_{0\le\xi\le\delta_0}|f'(\xi)|.
\]
Thus, this global existence criterion depends only on
the behavior of $f$ on $[0,\delta_0]$.
No estimates on $f$ outside this interval enter
the construction.
\end{remark}

\section{Model cases: power nonlinearities }\label{model}
\setcounter{equation}{0}

In this section, we apply the general results to power
nonlinearities, presenting blow-up before global existence.
We then specialize the time coefficient and illustrate
the resulting criteria on different underlying spaces.
Other nonlinearities are treated in
Section \ref{model-beyond-powers}.
Throughout this section, we assume \eqref{mainassumptions}
and retain the standing assumptions on the Dirichlet form
and the heat kernel.
In particular, the heat semigroup is symmetric and
sub-Markovian; conservativity is required only where
explicitly assumed.
The power nonlinearity is fixed at the beginning of
Subsection \ref{subsec:power-criteria}, where we verify
the assumptions needed for both blow-up and global existence.
We use the local existence and continuation result in
Proposition \ref{prop:continuation}.
For global existence, we also retain the kernel assumptions
and conventions of Theorem \ref{teo2}.
Thus $\overline{\mathcal K}$, defined in \eqref{supKernel},
is finite-valued and measurable on $(0,\infty)$, and
the pole $y_0$ and shift $\gamma>0$ are chosen so that
\eqref{columnsemigroup} and \eqref{columnbound} hold.
Smallness of the initial datum is understood through
a heat kernel majorization.

Spatial identities and inequalities are understood according
to the almost-everywhere convention fixed in
Section \ref{heats}.

\subsection{Power nonlinearities: general criteria}\label{subsec:power-criteria}

Throughout this section, we consider
\begin{equation}\label{case1}
f(v)=v^p
\qquad\text{for every }v\ge0,
\qquad p>1.
\end{equation}
This function is locally Lipschitz on $[0,\infty)$,
vanishes at zero, and is positive on $(0,\infty)$.
It therefore satisfies the assumptions on $f$ in
\eqref{mainassumptions}.
Moreover, it is nondecreasing and convex and satisfies
\eqref{Hf}. The function $J$ defined in \eqref{J} is
\begin{equation}\label{Jcase1}
J(a)=\frac{a^{1-p}}{p-1}
\qquad\text{for every }a>0,
\end{equation}
and hence
\[
J^{-1}(s)=[(p-1)s]^{-1/(p-1)}
\qquad\text{for every }s>0.
\]

\subsubsection{Blow-up from lower heat kernel bounds}

\begin{remark}\label{rem:model-finite-time}
For the nonlinearity \eqref{case1}, the sufficient condition
in Theorem \ref{teo1} becomes
\begin{equation}\label{power-finite-time}
(p-1)H(T)\|P_Tu_0\|_{L^\infty(X,\mu)}^{p-1}\ge1
\end{equation}
for some $T>0$.
Under this condition, the maximal solution blows up in finite
time and $T_{\max}\le T$.
In particular, this criterion applies to sufficiently large
multiples of a fixed nonnegative profile, as explained in
Remark \ref{rem:large-data-BU}, without requiring long-time
heat kernel estimates.
%
\end{remark}

The next criterion is expressed in terms of the volume of balls.
Using the actual volume growth can give a sharper conclusion
than using a doubling exponent.

\begin{theorem}\label{corollario2}
Assume \eqref{mainassumptions}, let $f$ be given by
\eqref{case1}, and suppose that the heat kernel satisfies
$(HK_\beta)$, as in Definition \ref{def:HK}.
Choose $x_0\in X$, $r>0$, and $\delta>0$ such that
\[
u_0(x)\ge\delta
\qquad\text{for every }x\in B_r(x_0).
\]
If
\begin{equation}\label{hpBUvolume}
\limsup_{T\to\infty}
\frac{H(T)^{1/(p-1)}}{V(x_0,T^{1/\beta})}=\infty,
\end{equation}
then the maximal nonnegative mild solution of
problem \eqref{problema} blows up in finite time.
In particular, the same conclusion holds if
\begin{equation}\label{hpBUcor2}
\limsup_{T\to\infty}
T^{-s/\beta}H(T)^{1/(p-1)}=\infty,
\end{equation}
where $s>0$ is any exponent for which \eqref{eq05} holds.
\end{theorem}

\begin{proof}
Set $B:=B_r(x_0)$ and $R:=T^{1/\beta}$.
For $T\ge r^\beta$ and $x\in B$,
\[
B_R(x)\subseteq B_{R+r}(x_0)\subseteq B_{2R}(x_0).
\]
Since $(HK_\beta)$ implies volume doubling,
\[
V(x,R)\le C_DV(x_0,R).
\]
Moreover, $d(x,y)\le2r$ for $x,y\in B$, and hence
\[
\begin{aligned}
g_\beta(c_2d(x,y),T)
&\ge
\exp\left[
-\left(\frac{(2c_2r)^\beta}{T}\right)^{1/(\beta-1)}
\right]\\
&\ge c_B>0
\end{aligned}
\]
for every $T\ge r^\beta$, with $c_B$ independent of $T$.
The lower estimate in $(HK_\beta)$ therefore gives
\[
K_T(x,y)\ge
\frac{c_1c_B}{C_DV(x_0,T^{1/\beta})}
\]
for every $T\ge r^\beta$ and almost every
$(x,y)\in B\times B$.
Consequently, there exists $A>0$ such that
\[
(P_Tu_0)(x)\ge
\frac{A}{V(x_0,T^{1/\beta})}
\]
for every sufficiently large $T$ and almost every $x\in B$.
Since $\mu(B)>0$, we obtain
\[
\|P_Tu_0\|_{L^\infty(X,\mu)}
\ge \frac{A}{V(x_0,T^{1/\beta})}.
\]
It follows that
\[
\begin{aligned}
(p-1)H(T)\|P_Tu_0\|_{L^\infty(X,\mu)}^{p-1}
&\ge
(p-1)A^{p-1}
\left(
\frac{H(T)^{1/(p-1)}}{V(x_0,T^{1/\beta})}
\right)^{p-1}.
\end{aligned}
\]
By \eqref{hpBUvolume}, the right-hand side exceeds $1$
for some sufficiently large $T$.
Remark \ref{rem:model-finite-time} yields finite-time blow-up.
Finally, \eqref{eq05} implies
\[
V(x_0,T^{1/\beta})\le CT^{s/\beta}
\]
for all sufficiently large $T$.
Thus \eqref{hpBUcor2} implies \eqref{hpBUvolume}.
\end{proof}

\begin{remark}\label{rem:large-time-lower}
For the power nonlinearity \eqref{case1}, the preceding
argument only requires a lower heat kernel bound at large
times on a fixed set carrying positive initial mass.
More precisely, suppose that $B$ has finite positive measure,
that $\int_Bu_0\,d\mu>0$, and that, for some $a>0$,
\[
K_T(x,y)\ge c_BT^{-a}
\]
for every sufficiently large $T$ and almost every
$(x,y)\in B\times B$.
Then
\[
\limsup_{T\to\infty}
T^{-a}H(T)^{1/(p-1)}=\infty
\]
implies finite-time blow-up.
Indeed,
\[
(P_Tu_0)(x)\ge
c_BT^{-a}\int_Bu_0\,d\mu
\]
for almost every $x\in B$, and
\eqref{power-finite-time} holds for some sufficiently large $T$.

This observation also applies in the discrete graph framework.
\end{remark}

For the autonomous equation, conservativity allows us to
include the critical exponent by combining the necessary
condition for global existence with an accumulation of mass.

\begin{theorem}[Blow-up at the critical exponent]
\label{thm:critical-volume}
Assume \eqref{mainassumptions}, let $h\equiv1$, and let
$f$ be given by \eqref{case1}.
Suppose that the heat kernel is conservative and satisfies
$(HK_\beta)$.
Assume that, for some $x_0\in X$, $D>0$, $C_V>0$, and $R_0>0$,
\begin{equation}\label{critical-volume-growth}
V(x_0,R)\le C_VR^D
\qquad\text{for every }R\ge R_0.
\end{equation}
If
\[
p=1+\frac{\beta}{D},
\]
then every nonnegative, nontrivial initial datum satisfying
the standing assumptions generates a maximal mild solution
that blows up in finite time.
\end{theorem}

\begin{proof}
Suppose, by contradiction, that a global nonnegative mild
solution $u$ exists.
We first show that global existence forces its total mass
to remain uniformly bounded in time, and then obtain a
contradiction from the accumulated reaction.

\noindent Fix $s\ge0$.
Since $h\equiv1$, the restart identity shows that
$t\mapsto u(s+t)$ is a mild solution with initial datum $u(s)$.
The argument of Lemma \ref{lemma2} applies to this restarted
solution: it requires only a bounded nonnegative initial
datum, without spatial continuity.
Together with \eqref{Jcase1}, it gives
\begin{equation}\label{critical-restarted-bound}
\|P_Tu(s)\|_{L^\infty(X,\mu)}
\le [(p-1)T]^{-1/(p-1)}
=(p-1)^{-D/\beta}T^{-D/\beta}
\end{equation}
for every $T>0$ and every $s\ge0$.

\noindent Set $R:=T^{1/\beta}$ and take
$R\ge\max\{1,R_0\}$.
For $x\in B_1(x_0)$ and $y\in B_R(x_0)$,
\[
d(x,y)<R+1\le2R,
\]
and
\[
V(x,R)
\le V(x_0,R+1)
\le V(x_0,2R)
\le C_V2^DR^D.
\]
The lower estimate in $(HK_\beta)$ therefore yields
\[
K_T(x,y)\ge A_1R^{-D}
\]
for almost every
$(x,y)\in B_1(x_0)\times B_R(x_0)$,
where $A_1>0$ is independent of $R$ and $s$.
Consequently,
\[
(P_Tu(s))(x)
\ge A_1R^{-D}
\int_{B_R(x_0)}u(y,s)\,d\mu(y)
\]
for almost every $x\in B_1(x_0)$.

\noindent Since $\mu(B_1(x_0))>0$, taking the essential supremum
and using \eqref{critical-restarted-bound} gives
\[
\begin{aligned}
A_1R^{-D}\int_{B_R(x_0)}u(y,s)\,d\mu(y)
&\le \|P_Tu(s)\|_{L^\infty(X,\mu)}\\
&\le (p-1)^{-D/\beta}R^{-D}.
\end{aligned}
\]
Letting $R\to\infty$ and applying monotone convergence,
we obtain
\begin{equation}\label{critical-uniform-mass}
\int_Xu(y,s)\,d\mu(y)\le C_*
\qquad\text{for every }s\ge0,
\end{equation}
where
\[
C_*:=A_1^{-1}(p-1)^{-D/\beta}.
\]
In particular, $u(s)\in L^1(X,\mu)$ for every $s\ge0$,
including $s=0$.

\noindent Since $u_0$ is nonnegative and nontrivial, choose $r>0$ such that
\[
m_0:=\int_{B_r(x_0)}u_0(y)\,d\mu(y)>0.
\]
For sufficiently large $t$, put $R:=t^{1/\beta}$,
so that $R\ge r$.
If $x\in B_R(x_0)$ and $y\in B_r(x_0)$, then
$d(x,y)<2R$.
Moreover, volume doubling gives
\[
V(x,R)\le V(x_0,2R)\le C_DV(x_0,R).
\]
The lower heat kernel estimate therefore implies
\[
K_t(x,y)\ge\frac{A_2}{V(x_0,R)}
\]
for almost every
$(x,y)\in B_R(x_0)\times B_r(x_0)$,
where $A_2>0$ is independent of $t$.

\noindent The mild formulation gives $u(t)\ge P_tu_0$, and hence
\[
u(x,t)\ge
\frac{A_2m_0}{V(x_0,t^{1/\beta})}
\]
for almost every $x\in B_{t^{1/\beta}}(x_0)$ and every
sufficiently large $t$.
Thus
\[
\begin{aligned}
\int_Xu(x,t)^p\,d\mu(x)
&\ge
(A_2m_0)^pV(x_0,t^{1/\beta})^{1-p}\\
&\ge
(A_2m_0)^pC_V^{1-p}t^{-D(p-1)/\beta}.
\end{aligned}
\]
Since $D(p-1)=\beta$, there exist $A_3>0$ and $t_0>0$ such that
\begin{equation}\label{critical-reaction-lower}
\int_Xu(x,t)^p\,d\mu(x)\ge\frac{A_3}{t}
\qquad\text{for every }t\ge t_0.
\end{equation}

\noindent Finally, symmetry and conservativity imply
\[
\int_XK_t(x,y)\,d\mu(x)=1
\]
for every $t>0$ and almost every $y$.
Integrating the mild formulation and using Tonelli's theorem
therefore gives
\[
\int_Xu(x,t)\,d\mu(x)
=
\int_Xu_0(x)\,d\mu(x)
+\int_0^t\int_Xu(x,s)^p\,d\mu(x)\,ds.
\]
By \eqref{critical-reaction-lower},
\[
\int_Xu(x,t)\,d\mu(x)
\ge
\int_Xu_0(x)\,d\mu(x)
+A_3\log\left(\frac{t}{t_0}\right)
\qquad\text{for every }t\ge t_0.
\]
This contradicts \eqref{critical-uniform-mass}.

\noindent Thus no global nonnegative mild solution exists.
Proposition \ref{prop:continuation} implies that
$T_{\max}<\infty$ and
\[
\lim_{t\uparrow T_{\max}}
\|u(\cdot,t)\|_{L^\infty(X,\mu)}=\infty.
\]
\end{proof}

\begin{remark}
Under the geometric hypotheses of
Theorem \ref{thm:critical-volume}, with $h\equiv1$ and
$f$ given by \eqref{case1}, combining that theorem with
Theorem \ref{corollario2} gives finite-time blow-up
for every admissible nonnegative, nontrivial initial datum
whenever
\[
1<p\le1+\frac{\beta}{D}.
\]
For the subcritical range, the volume bound
\eqref{critical-volume-growth} transfers to a center carrying
positive initial mass by inclusion of balls.
Conservativity is needed for the critical argument,
but not for the subcritical conclusion.
\end{remark}


\subsubsection{Global existence from upper heat kernel bounds}

We retain the nonlinearity \eqref{case1}, whose assumptions
required by Theorem \ref{teo2} have already been verified.

\begin{theorem}\label{corGE1}
Assume \eqref{mainassumptions}, let $f$ be given by
\eqref{case1}, and retain the kernel assumptions of
Theorem \ref{teo2}.
Suppose that
\begin{equation}\label{eq53case1}
I_{p,\gamma}:=
\int_0^\infty
h(t)\overline{\mathcal K}(t+\gamma)^{p-1}\,dt<\infty.
\end{equation}
Choose $M>0$ such that
\begin{equation}\label{smallMpower}
pM^{p-1}I_{p,\gamma}\le\frac12.
\end{equation}
If
\begin{equation}\label{smallinitialpower}
0\le u_0(x)\le\varepsilon K_\gamma(x,y_0),
\qquad
0<\varepsilon\le\frac M2,
\end{equation}
then problem \eqref{problema} admits a unique global
nonnegative mild solution satisfying
\[
0\le u(x,t)\le MK_{t+\gamma}(x,y_0)
\]
for every $t\ge0$ and almost every $x\in X$.
\end{theorem}

\begin{proof}
We apply Theorem \ref{teo2} with $\omega\equiv1$.
For the nonlinearity \eqref{case1},
\begin{equation}\label{thetaMcase1}
\theta_M(t)=
M^{p-1}\overline{\mathcal K}(t+\gamma)^{p-1},
\end{equation}
whereas
\begin{equation}\label{sigmaMcase1}
\sigma_M(t)=
pM^{p-1}\overline{\mathcal K}(t+\gamma)^{p-1}
=p\theta_M(t).
\end{equation}
Thus \eqref{smallMpower} gives
\[
\int_0^t h(s)\sigma_M(s)\,ds\le\frac12
\qquad\text{for every }t>0.
\]
Condition \eqref{LM} holds with $\delta=1/2$,
and \eqref{smallinitialpower} ensures
$\varepsilon\le(1-\delta)M$.
All the hypotheses of Theorem \ref{teo2} are therefore satisfied.
\end{proof}

A useful sufficient condition for \eqref{eq53case1} is
provided by a positive bottom of the spectrum. We write
\begin{equation}\label{positivebottom}
\lambda_1:=\inf\sigma(L)>0,
\qquad L=-\Delta.
\end{equation}
This concerns the bottom of the full spectrum, rather than
a positive gap above a zero eigenvalue.
The spectral condition gives decay on $L^2$.
A uniform heat kernel bound at one positive time allows us
to transfer this decay to the kernel.

\begin{proposition}\label{prop:spectralkernel}
Assume \eqref{positivebottom}. Let $\gamma>0$ and suppose that
\[
A_\gamma:=\overline{\mathcal K}(\gamma/2)<\infty.
\]
Then
\begin{equation}\label{spectralshiftbound}
\overline{\mathcal K}(t+\gamma)
\le A_\gamma e^{-\lambda_1t}
\qquad\text{for every }t\ge0.
\end{equation}
\end{proposition}

\begin{proof}
Set $a:=\gamma/2$.
The mass inequality and the essential bound on $K_a$ imply
\[
\|P_a\|_{L^2\to L^\infty}\le A_\gamma^{1/2}.
\]
Indeed, by the Cauchy--Schwarz inequality,
\[
\begin{aligned}
|P_av(x)|^2
&\le
\left(\int_XK_a(x,y)^2\,d\mu(y)\right)
\|v\|_{L^2(X,\mu)}^2\\
&\le A_\gamma\|v\|_{L^2(X,\mu)}^2
\end{aligned}
\]
for almost every $x$.
By self-adjointness and duality,
\[
\|P_a\|_{L^1\to L^2}\le A_\gamma^{1/2}.
\]
Moreover, the spectral theorem gives
\[
\|P_t\|_{L^2\to L^2}=e^{-\lambda_1t}.
\]
Factoring $P_{t+\gamma}=P_aP_tP_a$, we obtain
\[
\|P_{t+\gamma}\|_{L^1\to L^\infty}
\le A_\gamma e^{-\lambda_1t}.
\]
Since the kernel is nonnegative, this operator bound yields
\eqref{spectralshiftbound}.
\end{proof}

\begin{theorem}\label{thm:gap-power}
Assume \eqref{mainassumptions}, the kernel assumptions of
Theorem \ref{teo2}, and the hypotheses of
Proposition \ref{prop:spectralkernel}.
Let $f$ be given by \eqref{case1}, and suppose that
\begin{equation}\label{weightedhgap}
E_p:=
\int_0^\infty h(t)e^{-(p-1)\lambda_1t}\,dt<\infty.
\end{equation}
Choose $M>0$ such that
\[
pM^{p-1}A_\gamma^{p-1}E_p\le\frac12.
\]
If the initial datum satisfies \eqref{smallinitialpower},
then problem \eqref{problema} admits a unique global
nonnegative mild solution satisfying
\[
0\le u(x,t)\le MK_{t+\gamma}(x,y_0)
\]
for every $t\ge0$ and almost every $x\in X$.
Moreover,
\begin{equation}\label{gapdecay}
\|u(t)\|_{L^\infty(X,\mu)}
\le MA_\gamma e^{-\lambda_1t}
\qquad\text{for every }t\ge0.
\end{equation}
\end{theorem}

\begin{proof}
The power nonlinearity satisfies the assumptions on $f$
in Theorem \ref{corGE1}.
By \eqref{spectralshiftbound},
\[
I_{p,\gamma}
\le
A_\gamma^{p-1}
\int_0^\infty h(t)e^{-(p-1)\lambda_1t}\,dt
=
A_\gamma^{p-1}E_p.
\]
The global existence statement follows from
Theorem \ref{corGE1}.
The norm estimate follows from \eqref{columnbound}
and \eqref{spectralshiftbound}.
\end{proof}

\subsection{Power nonlinearities: the role of the time coefficient}

We now specialize the preceding criteria to polynomial and
exponential time coefficients, retaining the nonlinearity
\eqref{case1}.
As verified above, all the nonlinearity assumptions of the
blow-up and global existence results used below are fulfilled.

\subsubsection{Constant and polynomial time coefficients}

Consider
\[
h(t)=(1+t)^q,\qquad q>-1.
\]
This coefficient is continuous and strictly positive,
as required by \eqref{mainassumptions}.
The constant coefficient $h\equiv1$ corresponds to $q=0$.
We have
\[
H(T)=\frac{(1+T)^{q+1}-1}{q+1}
\sim\frac{T^{q+1}}{q+1}
\qquad\text{as }T\to\infty.
\]

\begin{example}[Blow-up under polynomial volume growth]
Assume $(HK_\beta)$ and suppose that, for some $d>0$,
\[
V(x_0,R)\le CR^d
\]
for all sufficiently large $R$, where $x_0$ is chosen
as in Theorem \ref{corollario2}.
Then, for some $C_0>0$,
\[
\frac{H(T)^{1/(p-1)}}{V(x_0,T^{1/\beta})}
\ge C_0T^{(q+1)/(p-1)-d/\beta}
\]
for all sufficiently large $T$.
Theorem \ref{corollario2} therefore gives finite-time blow-up
whenever
\[
1<p<1+\frac{\beta(q+1)}d.
\]
If $q=0$ and the heat kernel is conservative,
Theorem \ref{thm:critical-volume} also covers equality,
so that the blow-up range becomes
\[
1<p\le1+\frac{\beta}{d}.
\]
\end{example}

\begin{example}[Global existence under polynomial heat kernel decay]
Suppose that, for some $a>0$ and $C>0$,
\[
\overline{\mathcal K}(t+\gamma)
\le C(1+t)^{-a}
\qquad\text{for every }t\ge0.
\]
For $h(t)=(1+t)^q$, condition \eqref{eq53case1} follows from
\[
\int_0^\infty(1+t)^{q-a(p-1)}\,dt<\infty.
\]
Consequently, Theorem \ref{corGE1} gives global solutions
for sufficiently small initial data satisfying
\eqref{smallinitialpower} whenever
\[
p>1+\frac{q+1}{a}.
\]
In particular, if the volume growth is of order $R^d$,
with $d>0$, and the corresponding uniform heat kernel decay
is of order $t^{-d/\beta}$, the blow-up and global existence
criteria have the same threshold
\[
1+\frac{\beta(q+1)}d.
\]
For general $q>-1$, the preceding asymptotic criteria do not
settle the equality case.
For $q=0$, Theorem \ref{thm:critical-volume} gives blow-up
at equality whenever its additional hypotheses hold.
The finite-time criterion in Remark \ref{rem:model-finite-time}
remains applicable to sufficiently large data for every $q>-1$.
\end{example}


\subsubsection{Exponentially increasing time coefficients}

We now assume
\begin{equation}\label{eq:exponential-coefficients}
h(t)=e^{qt},\qquad q>0.
\end{equation}
The coefficient $h$ satisfies \eqref{mainassumptions},
and the nonlinearity remains the one in \eqref{case1}.
In this case,
\begin{equation}\label{eq:JH-exponential}
J(a)=\frac{a^{1-p}}{p-1},
\qquad
H(T)=\frac{e^{qT}-1}{q},
\end{equation}
and
\begin{equation}\label{eq:Jinverse-exponential}
\begin{aligned}
J^{-1}(H(T))
&=
\left[\frac{p-1}{q}(e^{qT}-1)\right]^{-1/(p-1)}\\
&\sim
\left(\frac{q}{p-1}\right)^{1/(p-1)}
e^{-qT/(p-1)}
\qquad\text{as }T\to\infty.
\end{aligned}
\end{equation}

The first two results concern blow-up, respectively under
Gaussian estimates and under an exponential lower bound
governed by the spectral bottom.

\begin{theorem}\label{thm:gaussian-exponential-blowup}
Assume \eqref{mainassumptions},
\eqref{case1}, and \eqref{eq:exponential-coefficients}.
Suppose that the heat kernel satisfies the bilateral
Gaussian estimates $(HK_2)$.
Then every nonnegative, nontrivial initial datum satisfying
the standing assumptions generates a maximal mild solution
that blows up in finite time.
\end{theorem}

\begin{proof}
The condition $(HK_2)$ implies volume doubling.
Let $s>0$ be an exponent in \eqref{eq05}.
By \eqref{eq:JH-exponential},
\[
T^{-s/2}H(T)^{1/(p-1)}
=
T^{-s/2}
\left(\frac{e^{qT}-1}{q}\right)^{1/(p-1)}
\longrightarrow\infty
\qquad\text{as }T\to\infty.
\]
Thus \eqref{hpBUcor2} holds with $\beta=2$,
and Theorem \ref{corollario2} applies.
\end{proof}

\begin{remark}
The same argument applies under $(HK_\beta)$ for any
admissible $\beta$.
In particular, exponential growth of the time coefficient
excludes global existence for every $p>1$ in this setting,
however small the nontrivial initial datum may be.
The argument uses a lower heat kernel bound and polynomial
volume growth; a Gaussian upper bound alone is insufficient.
\end{remark}

For the spectral blow-up criterion, suppose that
$\lambda_1:=\inf\sigma(L)>0$ and that, for every relatively
compact ball $B\subset X$ and every $\eta>0$, there exist
$c_{B,\eta}>0$ and $T_{B,\eta}>0$ such that
\begin{equation}\label{eq:spectral-exponential-lower}
\operatorname*{ess\,inf}_{(x,y)\in B\times B}K_t(x,y)
\ge c_{B,\eta}e^{-(\lambda_1+\eta)t}
\qquad\text{for every }t\ge T_{B,\eta}.
\end{equation}
The essential infimum is taken with respect to $\mu\otimes\mu$.
Unlike \eqref{E2}, this condition is imposed on every
relatively compact ball, so that the lower bound can be used
on a ball carrying positive initial mass.

\begin{theorem}\label{thm:spectral-exponential-blowup}
Assume \eqref{mainassumptions},
\eqref{case1}, and \eqref{eq:exponential-coefficients}.
Suppose that $\lambda_1>0$,
\eqref{eq:spectral-exponential-lower} holds, and
\begin{equation}\label{eq:spectral-exponential-blowup}
q>(p-1)\lambda_1.
\end{equation}
Then every nonnegative, nontrivial initial datum satisfying
the standing assumptions generates a maximal mild solution
that blows up in finite time.
\end{theorem}

\begin{proof}
Choose a relatively compact ball $B=B_r(x_0)$ and
$\delta>0$ such that $u_0\ge\delta$ on $B$.
By \eqref{eq:spectral-exponential-blowup}, choose $\eta>0$ with
\[
\lambda_1+\eta<\frac q{p-1}.
\]
For every sufficiently large $T$,
\eqref{eq:spectral-exponential-lower} gives
\[
\frac{\underline{\mathcal K}(T)}{J^{-1}(H(T))}
\ge
c_{B,\eta}e^{-(\lambda_1+\eta)T}
\left[\frac{p-1}{q}(e^{qT}-1)\right]^{1/(p-1)}.
\]
The right-hand side tends to infinity, since
$q/(p-1)-\lambda_1-\eta>0$.
Thus \eqref{hpBU} holds.
The nonlinearity \eqref{case1} satisfies
\eqref{mainassumptions}, \eqref{Hf}, convexity, and monotonicity,
so Corollary \ref{corteo1} applies.
\end{proof}

We next give the complementary global existence result.

\begin{theorem}\label{thm:spectral-exponential-existence}
Assume \eqref{mainassumptions}, the kernel assumptions of
Theorem \ref{teo2}, the hypotheses of
Proposition \ref{prop:spectralkernel},
\eqref{case1}, and \eqref{eq:exponential-coefficients}.
Suppose that
\begin{equation}\label{eq:spectral-exponential-existence}
q<(p-1)\lambda_1.
\end{equation}
Choose $M>0$ such that
\begin{equation}\label{eq:spectral-exponential-smallness}
pM^{p-1}A_\gamma^{p-1}
\le\frac{(p-1)\lambda_1-q}{2}.
\end{equation}
If
\begin{equation}\label{eq:spectral-exponential-data}
0\le u_0(x)\le\varepsilon K_\gamma(x,y_0),
\qquad
0<\varepsilon\le\frac M2,
\end{equation}
then problem \eqref{problema} admits a unique global
nonnegative mild solution satisfying
\begin{equation}\label{eq:spectral-exponential-bound}
0\le u(x,t)\le MK_{t+\gamma}(x,y_0)
\end{equation}
for every $t\ge0$ and almost every $x\in X$.
Moreover,
\begin{equation}\label{eq:spectral-exponential-decay}
\|u(t)\|_{L^\infty(X,\mu)}
\le MA_\gamma e^{-\lambda_1t}
\qquad\text{for every }t\ge0.
\end{equation}
\end{theorem}

\begin{proof}
Condition \eqref{eq:spectral-exponential-existence} gives
\[
E_p
=
\int_0^\infty e^{-((p-1)\lambda_1-q)t}\,dt
=
\frac1{(p-1)\lambda_1-q}.
\]
In particular,
\begin{equation}\label{eq:spectral-exponential-integral}
I_{p,\gamma}
\le
\frac{A_\gamma^{p-1}}{(p-1)\lambda_1-q}
<\infty.
\end{equation}
The nonlinearity is the power considered in
Theorem \ref{thm:gap-power}, and
\eqref{eq:spectral-exponential-smallness} is precisely
the required smallness condition.
The conclusion follows from that theorem.
\end{proof}

\begin{remark}\label{rem:spectral-exponential-threshold}
Under the respective heat kernel assumptions, for
the nonlinearity \eqref{case1} and $h(t)=e^{qt}$,
the spectral criteria give
finite-time blow-up for every nonnegative, nontrivial
initial datum when
\[
1<p<1+\frac q{\lambda_1},
\]
and global existence for sufficiently small initial data when
\[
p>1+\frac q{\lambda_1}.
\]
At the borderline $q=(p-1)\lambda_1$, these exponential
estimates alone do not decide whether sufficiently small
nontrivial data generate global solutions.
Blow-up for sufficiently large data is still covered by
Remark \ref{rem:model-finite-time}.
Additional information on the heat kernel can settle
small-data existence at the borderline.
For example, suppose that
\[
\overline{\mathcal K}(t+\gamma)
\le
C_\gamma e^{-\lambda_1t}(1+t)^{-\alpha}
\qquad\text{for every }t\ge0
\]
for some $\alpha>0$.
When $q=(p-1)\lambda_1$,
\[
I_{p,\gamma}
\le
C_\gamma^{p-1}
\int_0^\infty(1+t)^{-\alpha(p-1)}\,dt
<\infty
\]
whenever $\alpha(p-1)>1$.
Under this additional assumption, Theorem \ref{corGE1}
gives global existence for sufficiently small initial data
dominated by a heat kernel column.
\end{remark}

\subsection{Examples of underlying spaces}

We now illustrate the geometric hypotheses used above.
Throughout this subsection, we retain the nonlinearity
\eqref{case1}, and, unless stated otherwise, we take
$h\equiv1$.
This coefficient satisfies \eqref{mainassumptions}.
The same holds when $h(t)=e^{qt}$, with $q>0$.
The nonlinearity assumptions of the blow-up and global
existence results invoked below have already been verified.
All initial data satisfy the standing assumptions.
All global existence conclusions concern sufficiently small
initial data satisfying the prescribed heat kernel majorization.

\subsubsection{Euclidean space and Riemannian manifolds}

\begin{example}[Euclidean space]
Let $X=\mathbb R^n$, $n\ge1$, with the usual Laplacian.
We first consider $h\equiv1$.
Since $\beta=2$, $V(x,R)=c_nR^n$, and $H(T)=T$,
condition \eqref{hpBUvolume} becomes
\[
\limsup_{T\to\infty}T^{-n/2+1/(p-1)}=\infty.
\]
Thus Theorem \ref{corollario2} gives finite-time blow-up
for every nonnegative, nontrivial initial datum satisfying
the standing assumptions if
\[
1<p<1+\frac2n.
\]
At the critical exponent $p=1+2/n$, the same conclusion
follows from Theorem \ref{thm:critical-volume}, with
$\beta=2$ and $D=n$, since the Euclidean heat kernel is
conservative.
Thus the blow-up range is
\[
1<p\le1+\frac2n.
\]
For global existence, still with $h\equiv1$, we use
\[
\overline{\mathcal K}(t)=(4\pi t)^{-n/2}.
\]
Hence
\[
I_{p,\gamma}
=
(4\pi)^{-n(p-1)/2}
\int_0^\infty(t+\gamma)^{-n(p-1)/2}\,dt
<\infty
\]
if and only if
\[
p>1+\frac2n.
\]
For these exponents, Theorem \ref{corGE1} gives global
existence for sufficiently small initial data satisfying
the prescribed heat kernel majorization.
Together, these conclusions recover the Fujita threshold
(see e.g. \cite{Fujita, H, KST, QS, Sug}).
Every nonnegative, nontrivial initial datum satisfying
the standing assumptions leads to finite-time blow-up
for $1<p\le1+2/n$, whereas sufficiently small initial data
satisfying the heat kernel majorization generate global
solutions for $p>1+2/n$.
If instead
\[
h(t)=e^{qt},\qquad q>0,
\]
then $h$ is continuous and strictly positive, as required by
\eqref{mainassumptions}.
Theorem \ref{thm:gaussian-exponential-blowup} therefore
gives finite-time blow-up for every $p>1$ and every
nonnegative, nontrivial initial datum satisfying the
standing assumptions.
\end{example}

\begin{example}[Nonnegative Ricci curvature]
Let $X$ be a complete connected Riemannian manifold of
dimension $n$ with nonnegative Ricci curvature, and let
$\Delta$ be its Laplace--Beltrami operator.
We first take $h\equiv1$.
The heat kernel is conservative and satisfies $(HK_2)$,
as recalled in Theorem \ref{thm:ricci}.
For each fixed $x_0$,
\[
V(x_0,R)\le CR^n.
\]
Theorem \ref{corollario2} gives blow-up for $1<p<1+2/n$.
Since the heat kernel is conservative,
Theorem \ref{thm:critical-volume}, with $\beta=2$ and $D=n$,
also covers the critical exponent.
Thus every nonnegative, nontrivial initial datum satisfying
the standing assumptions generates finite-time blow-up whenever
\[
1<p\le1+\frac2n.
\]
For global existence, still with $h\equiv1$,
assume additionally that, for some $d>0$,
\[
V(x,R)\ge cR^d
\qquad\text{uniformly in }x,\quad R\ge1.
\]
The Gaussian upper bound then gives
\[
\overline{\mathcal K}(t)\le Ct^{-d/2}
\qquad\text{for }t\ge1.
\]
The required finiteness at smaller positive times also holds.
Indeed, Bishop--Gromov volume comparison and the lower bound
on $V(x,1)$ imply
\[
V(x,r)\ge r^nV(x,1)\ge cr^n
\qquad\text{for }0<r\le1,
\]
and hence
\[
\overline{\mathcal K}(t)\le Ct^{-n/2}
\qquad\text{for }0<t\le1.
\]
Choosing $\gamma\ge1$, Theorem \ref{corGE1} therefore gives
global existence for sufficiently small initial data
satisfying the heat kernel majorization whenever
\[
p>1+\frac2d.
\]
If the large-scale volume growth is uniformly comparable
to $R^d$, Theorems \ref{corollario2} and
\ref{thm:critical-volume} give blow-up for
$1<p\le1+2/d$, while sufficiently small data satisfying
the heat kernel majorization generate global solutions
for $p>1+2/d$.
Concerning blow-up for $h\equiv1$, our conclusions are
in agreement with the results of \cite{MaMoPu, Zhang},
obtained by entirely different methods.
Those works, however, do not address global existence in general.
If instead
\[
h(t)=e^{qt},\qquad q>0,
\]
the assumptions on $h$ required by
Theorem \ref{thm:gaussian-exponential-blowup} are again
satisfied.
That theorem gives finite-time blow-up for every $p>1$
and every nonnegative, nontrivial initial datum satisfying
the standing assumptions, without the additional lower
volume growth assumption.
\end{example}

\begin{example}[Hyperbolic space and negatively curved manifolds]
We first consider $h\equiv1$, retaining the nonlinearity
\eqref{case1}.
Both the global existence and blow-up criteria below are
applicable under their respective heat kernel and data assumptions.

Let $X=\mathbb H^n$, $n\ge2$, with sectional curvature $-1$.
Its heat kernel is conservative and uniformly bounded at
every positive time, and
\[
\lambda_1(\mathbb H^n)=\frac{(n-1)^2}{4}>0.
\]
Theorem \ref{thm:gap-power} therefore gives global solutions
for every $p>1$ and sufficiently small initial data satisfying
the heat kernel majorization.
For every such exponent, sufficiently large multiples of
a fixed nonnegative, nontrivial profile satisfying the
standing assumptions lead to blow-up by
Remark \ref{rem:model-finite-time}.
More generally, with the same choices of $f$ and $h$,
let $X$ be a complete simply connected
Riemannian manifold of dimension $n\ge2$ satisfying
\[
-b^2\le\operatorname{Sec}_X\le-a^2<0
\]
for some $0<a\le b$.
Its heat semigroup is conservative, its heat kernel has
uniform bounds at positive times, and
\[
\lambda_1(X)\ge\frac{(n-1)^2a^2}{4}>0.
\]
The same small-data existence conclusion holds for every
$p>1$, while the finite-time criterion applies to
sufficiently large data.

For these spaces, if instead
\[
h(t)=e^{qt},\qquad q>0,
\]
the assumptions on $h$ in
Theorem \ref{thm:spectral-exponential-existence} are satisfied.
That theorem gives global existence for sufficiently small
initial data satisfying the heat kernel majorization whenever
\[
q<(p-1)\lambda_1.
\]
These conclusions are in agreement with the results
of \cite{BPT, Punzo}.
Further examples are provided by complete manifolds with
positive spectral bottom and a uniform Faber--Krahn inequality
supplying the required heat kernel bounds; see
\cite{GrilloMeglioliPunzo2023}.
For the nonlinearity \eqref{case1}, the same global existence
criteria apply with $h\equiv1$, or with $h(t)=e^{qt}$ when
$q<(p-1)\lambda_1$.
\end{example}

\subsubsection{Metric and discrete graphs}

Metric graphs are considered with the path metric, edge-length
measure, and Kirchhoff Laplacian introduced in Section \ref{mf}.
For the discrete examples below, we use instead the discrete
Dirichlet form framework of that section, replacing the geometric
assumption \eqref{mainassumptions}(i).
The local existence, continuation, finite-time blow-up, and
weighted global existence arguments remain applicable:
they use the semigroup properties, the kernel representation,
and the assumptions on $f$ and $h$, rather than the geodesic
structure of the underlying space.
Spatial integrals become sums, and almost-everywhere identities
become pointwise identities.
We do not impose the all-time estimate $(HK_\beta)$ on
the discrete graphs.

\begin{example}[Metric lattices]\label{ex:metric-lattice}
Let $X$ be the metric realization of the nearest-neighbour
lattice $\mathbb Z^d$, $d\ge1$, with every edge identified
with an interval of length one.
Equip $X$ with edge-length measure and the Kirchhoff Laplacian.
The heat semigroup is conservative, and the volume satisfies
\[
V(x,R)\asymp
\begin{cases}
R, & 0<R\le1,\\
R^d, & R\ge1,
\end{cases}
\]
uniformly in $x$.
Global volume doubling and a scale-invariant Poincar\'e
inequality yield $(HK_2)$; see \cite{Haeseler}.
Consequently,
\[
\overline{\mathcal K}(t)\asymp
\begin{cases}
t^{-1/2}, & 0<t\le1,\\
t^{-d/2}, & t\ge1.
\end{cases}
\]
For the nonlinearity \eqref{case1} and $h\equiv1$,
Theorems \ref{corollario2} and \ref{thm:critical-volume}
give finite-time blow-up whenever
\[
1<p\le1+\frac2d.
\]
Theorem \ref{corGE1} gives global existence for sufficiently
small initial data satisfying the heat kernel majorization
whenever
\[
p>1+\frac2d.
\]
These results are in accordance with those established in \cite{PT}.
For $h(t)=(1+t)^q$, $q>-1$, the corresponding strict
conditions are
\[
1<p<1+\frac{2(q+1)}d
\quad\text{for blow-up},
\qquad
p>1+\frac{2(q+1)}d
\quad\text{for small-data global existence}.
\]
Equality is included in the blow-up range when $q=0$;
the preceding criteria do not settle it for $q\ne0$.

For $h(t)=e^{qt}$, $q>0$, finite-time blow-up holds
for every $p>1$.

Thus the thresholds depend on the large-scale dimension $d$,
although the graph is locally one-dimensional.
\end{example}

\begin{example}[Discrete lattices]\label{ex:discrete-lattice}
Let $X=\mathbb Z^d$, $d\ge1$, with counting measure and
the unnormalized nearest-neighbour Laplacian
\[
(\Delta u)(x)=\sum_{|y-x|_1=1}\bigl(u(y)-u(x)\bigr).
\]
The associated semigroup is symmetric and conservative.
Its kernel satisfies (see e.g. \cite{Barlow, PuZu2})
\[
0\le K_t(x,y)\le1
\qquad(t\ge0),
\]
and
\begin{equation}\label{eq:discrete-lattice-upper}
\sup_{x,y\in\mathbb Z^d}K_t(x,y)
\le C(1+t)^{-d/2}.
\end{equation}
Moreover, for every fixed $A>0$, there exists $c_A>0$ such that
\begin{equation}\label{eq:discrete-lattice-lower}
K_t(x,y)\ge c_A t^{-d/2}
\qquad
\text{if }t\ge1,\quad |x-y|_1\le A\sqrt t.
\end{equation}

For $h(t)=(1+t)^q$, $q>-1$, the lower bound on a fixed
finite set and Remark \ref{rem:large-time-lower} give
finite-time blow-up when
\[
1<p<1+\frac{2(q+1)}d.
\]
By \eqref{eq:discrete-lattice-upper},
\[
I_{p,\gamma}
\le C_\gamma^{p-1}
\int_0^\infty(1+t)^{q-d(p-1)/2}\,dt<\infty
\]
when
\[
p>1+\frac{2(q+1)}d.
\]
The weighted global existence argument therefore gives
global solutions for sufficiently small initial data
dominated by a kernel column.

In the special case $h\equiv 1$, such results are in line with those established in \cite{LinWu1,MoPuSo,PuZu2}.
%

Finally, if $h(t)=e^{qt}$, $q>0$, the polynomial lower
bound on a fixed finite set gives blow-up for every $p>1$.
\end{example}

\begin{example}[Regular discrete trees]\label{ex:regular-tree}
Let $X=\mathbb T_k$ be the infinite $k$-regular tree,
$k\ge3$, with counting measure and
\[
(\Delta u)(x)=\sum_{y\sim x}\bigl(u(y)-u(x)\bigr).
\]
The semigroup is symmetric and conservative, its kernel
is bounded by one, and
\[
\lambda_1=k-2\sqrt{k-1}>0.
\]
The standard heat kernel estimates on homogeneous trees
also give, on every fixed finite ball $B$,
\[
\inf_{x,y\in B}K_t(x,y)
\ge c_{B,\eta}e^{-(\lambda_1+\eta)t}
\]
for every $\eta>0$ and all sufficiently large $t$;
see \cite{GrilloMeglioliPunzoGraphs}.

For the nonlinearity \eqref{case1}, the spectral global
existence criterion gives global solutions for sufficiently
small initial data for every $p>1$ whenever $h$ has at most
polynomial growth.

If $h(t)=e^{qt}$, $q>0$, the spectral criteria give
\[
q<(p-1)\lambda_1
\quad\text{for small-data global existence},
\]
and
\[
q>(p-1)\lambda_1
\quad\text{for blow-up of every admissible nontrivial datum}.
\]
These strict conditions do not settle equality.
Sufficiently large multiples of a fixed nonnegative,
nontrivial profile generate blow-up in all these regimes. Such results are in agreement with those established in \cite{GrilloMeglioliPunzoGraphs}.
\end{example}

\section{Model cases: nonlinearities beyond powers}
\label{model-beyond-powers}
\setcounter{equation}{0}

We extend the preceding applications to sums of powers and
exponential nonlinearities.
As in Section \ref{model}, we first establish general criteria,
then specialize the time coefficient, and finally consider
examples of underlying spaces.
Throughout this section, we retain the standing assumptions
\eqref{mainassumptions} and the semigroup framework fixed above.
All global existence statements also retain the kernel
assumptions and conventions of Theorem \ref{teo2}, including
\eqref{columnsemigroup} and \eqref{columnbound}.
Small initial data are understood through the majorization
\[
0\le u_0(x)\le\varepsilon K_\gamma(x,y_0)
\]
for sufficiently small $\varepsilon>0$.
All blow-up conclusions concern admissible nonnegative,
nontrivial initial data.

We consider the following four nonlinearities, each defined
for every $v\ge0$:
\begin{equation}\label{eq:sum-powers}
f(v)=v^p+v^r,
\qquad 1<p<r;
\end{equation}
\begin{equation}\label{eq:linear-plus-power}
f(v)=a_0v+bv^p,
\qquad a_0,b>0,\quad p>1;
\end{equation}
\begin{equation}\label{eq:exponential-power-nonlinearity}
f(v)=e^{v^p}-1,
\qquad p>1;
\end{equation}
\begin{equation}\label{eq:exponential-linear-nonlinearity}
f(v)=c(e^{bv}-1),
\qquad b,c>0.
\end{equation}
The case $a_0=0$ in \eqref{eq:linear-plus-power} reduces
to the power nonlinearity of Section \ref{model},
with $h$ replaced by $bh$.
The applications below are organized according to the
behavior of the nonlinearity near zero.
For \eqref{eq:sum-powers} and
\eqref{eq:exponential-power-nonlinearity},
\[
f(v)\sim v^p
\qquad\text{as }v\downarrow0,
\]
and the relevant criteria involve the exponent $p$
and the decay of the heat kernel.
For \eqref{eq:linear-plus-power} and
\eqref{eq:exponential-linear-nonlinearity},
\[
f(v)\sim av
\qquad\text{as }v\downarrow0,
\]
where $a=a_0$ and $a=cb$, respectively.
In these cases, the criteria involve the interaction
between the linear slope $a$, the time coefficient,
and the spectral bottom.
The properties needed for these applications are verified
below in two propositions, one for sums of powers
and one for exponential nonlinearities.

\subsection{General criteria}

We begin with the extension of the power blow-up arguments.

\begin{remark}\label{rem:general-power-blowup}
Suppose that $f$ is nondecreasing and convex and that,
for some $p>1$,
\[
f(v)\ge v^p
\qquad\text{for every }v\ge0.
\]
Then \eqref{Hf} holds and
\[
J(a)=\int_a^\infty\frac{dv}{f(v)}
\le\frac{a^{1-p}}{p-1}
\qquad(a>0).
\]
Hence \eqref{power-finite-time} remains sufficient for blow-up
by Theorem \ref{teo1}.
In particular, the arguments of Theorem \ref{corollario2},
Remark \ref{rem:large-time-lower}, and the exponential-time
blow-up results extend to $f$ under their respective
geometric and time-coefficient assumptions.
This also includes the critical autonomous case of
Theorem \ref{thm:critical-volume}.
Indeed, the bound on $J$ yields the same restarted estimate
\eqref{critical-restarted-bound}, and hence the uniform
mass bound \eqref{critical-uniform-mass}.
Conservativity and the mild formulation give
\[
\begin{aligned}
\int_Xu(x,t)\,d\mu(x)
&=\int_Xu_0(x)\,d\mu(x)
  +\int_0^t\int_Xf(u(x,s))\,d\mu(x)\,ds\\
&\ge\int_Xu_0(x)\,d\mu(x)
  +\int_0^t\int_Xu(x,s)^p\,d\mu(x)\,ds.
\end{aligned}
\]
Since \eqref{critical-reaction-lower} still follows from
$u(t)\ge P_tu_0$, the same logarithmic contradiction proves
finite-time blow-up.
\end{remark}

For global existence, the following criterion requires
quantitative control only near zero.

\begin{proposition}\label{rem:localpower}
Assume the standing global existence hypotheses.
Suppose that there exist $r_0>0$, $C_f>0$, and $p>1$ such that
\begin{equation}\label{eq:local-power-control}
\operatorname{Lip}(f;[0,r])
\le C_fr^{p-1}
\qquad\text{for every }0<r\le r_0,
\end{equation}
where $\operatorname{Lip}(f;[0,r])$ denotes the optimal
Lipschitz constant of $f$ on $[0,r]$.
If
\[
I_{p,\gamma}
=
\int_0^\infty
h(t)\overline{\mathcal K}(t+\gamma)^{p-1}\,dt<\infty,
\]
then sufficiently small initial data generate a unique
global nonnegative mild solution.
More precisely, choose $M>0$ such that
\[
M\overline{\mathcal K}(\gamma)\le r_0,
\qquad
C_fM^{p-1}I_{p,\gamma}\le\frac12.
\]
For initial data satisfying \eqref{smallinitialpower},
the solution obeys
\[
0\le u(x,t)\le MK_{t+\gamma}(x,y_0)
\]
for every $t\ge0$ and almost every $x\in X$.
\end{proposition}

\begin{proof}
The semigroup property and sub-Markovianity imply
\[
\begin{aligned}
K_{t+s}(x,y)
&=\int_XK_s(x,z)K_t(z,y)\,d\mu(z)\\
&\le
\overline{\mathcal K}(t)
\int_XK_s(x,z)\,d\mu(z)
\le\overline{\mathcal K}(t)
\end{aligned}
\]
for every $s,t>0$ and almost every $(x,y)$.
Consequently,
\[
\overline{\mathcal K}(t+\gamma)
\le\overline{\mathcal K}(\gamma)
\qquad(t\ge0).
\]
With $\omega\equiv1$, the chosen $M$ therefore gives
\[
R_M(t)\le r_0,
\qquad
\sigma_M(t)
\le
C_fM^{p-1}\overline{\mathcal K}(t+\gamma)^{p-1}.
\]
Thus \eqref{LM} holds with $\delta=1/2$,
and Theorem \ref{teo2} applies.
\end{proof}

A positive spectral bottom also allows a nonzero linear part
near the origin.
The next criterion uses a time-dependent weight and does not
require convexity or monotonicity.

\begin{theorem}\label{thm:gap-general}
Assume the standing global existence hypotheses and
the hypotheses of Proposition \ref{prop:spectralkernel}.
Suppose that there exist $r_0>0$ and $\ell\ge0$ such that
\begin{equation}\label{locallipschitzgap}
\operatorname{Lip}(f;[0,r_0])\le\ell.
\end{equation}
Assume that, for some $\nu>\ell$,
\begin{equation}\label{eq:general-time-gap}
B_\nu:=
\sup_{t\ge0}e^{\nu H(t)-\lambda_1t}<\infty.
\end{equation}
Then sufficiently small initial data generate a unique
global nonnegative mild solution.

More precisely, choose $\delta\in(\ell/\nu,1)$ and
$M,\varepsilon>0$ such that
\[
MA_\gamma B_\nu\le r_0,
\qquad
\varepsilon\le(1-\delta)M.
\]
If
\[
0\le u_0(x)\le\varepsilon K_\gamma(x,y_0),
\]
then
\[
0\le u(x,t)\le
Me^{\nu H(t)}K_{t+\gamma}(x,y_0)
\]
for every $t\ge0$ and almost every $x\in X$.
Moreover,
\[
\|u(t)\|_{L^\infty(X,\mu)}
\le MA_\gamma e^{\nu H(t)-\lambda_1t}.
\]
\end{theorem}

\begin{proof}
Apply Theorem \ref{teo2} with
\[
\omega(t)=e^{\nu H(t)}.
\]
This weight satisfies \eqref{omega}.
By \eqref{spectralshiftbound},
\[
R_M(t)
=
Me^{\nu H(t)}\overline{\mathcal K}(t+\gamma)
\le
MA_\gamma e^{\nu H(t)-\lambda_1t}
\le r_0.
\]
Hence $\sigma_M(t)\le\ell$, and
\[
\begin{aligned}
\int_0^t\sigma_M(s)h(s)\omega(s)\,ds
&\le
\ell\int_0^t h(s)e^{\nu H(s)}\,ds\\
&=\frac{\ell}{\nu}
  \bigl(e^{\nu H(t)}-1\bigr)
\le\delta\omega(t).
\end{aligned}
\]
Condition \eqref{LM} follows.
Since $\omega(0)=1$, the initial majorization is precisely
the one required by Theorem \ref{teo2}.
The norm estimate follows from \eqref{columnbound}
and \eqref{spectralshiftbound}.
\end{proof}

We now verify the properties needed to apply these criteria
to the four nonlinearities introduced above.

\begin{remark}[Sums of powers]
\label{prop:properties-sums}
The nonlinearities in \eqref{eq:sum-powers} and
\eqref{eq:linear-plus-power} satisfy the standing assumptions
on $f$. Indeed, both belong to $C^1([0,\infty))$, vanish at
zero, are positive on $(0,\infty)$, and are nondecreasing
and convex. Their power growth at infinity also ensures
\eqref{Hf}.

For \eqref{eq:sum-powers}, we have
\[
f(v)\ge v^p\qquad(v\ge0),
\]
and, for every $r_0>0$,
\[
\operatorname{Lip}(f;[0,\rho])
=p\rho^{p-1}+r\rho^{r-1}
\le \bigl(p+rr_0^{r-p}\bigr)\rho^{p-1}
\qquad(0<\rho\le r_0).
\]
Thus Remark \ref{rem:general-power-blowup} and
Proposition \ref{rem:localpower} apply with exponent $p$.

For \eqref{eq:linear-plus-power}, differentiation gives
\begin{equation}\label{eq:linear-power-lip}
\operatorname{Lip}(f;[0,r_0])
=a_0+bp r_0^{p-1},
\end{equation}
while direct integration yields
\begin{equation}\label{eq:linear-power-J}
J(z)
=
\frac1{a_0(p-1)}
\log\left(1+\frac{a_0}{bz^{p-1}}\right).
\end{equation}
Consequently,
\begin{equation}\label{eq:linear-power-Jinverse}
\begin{aligned}
J^{-1}(s)
&=
\left[
\frac{a_0}{b\bigl(e^{a_0(p-1)s}-1\bigr)}
\right]^{1/(p-1)}\\
&\sim
\left(\frac{a_0}{b}\right)^{1/(p-1)}e^{-a_0s}
\qquad\text{as }s\to\infty.
\end{aligned}
\end{equation}
\end{remark}

\begin{remark}[Exponential nonlinearities]
\label{prop:properties-exponentials}
The nonlinearities in \eqref{eq:exponential-power-nonlinearity}
and \eqref{eq:exponential-linear-nonlinearity} satisfy the
standing assumptions on $f$. Indeed, both belong to
$C^1([0,\infty))$, vanish at zero, are positive on
$(0,\infty)$, and are nondecreasing and convex.
Their exponential growth at infinity also ensures \eqref{Hf}.

For \eqref{eq:exponential-power-nonlinearity}, the inequality
$e^z-1\ge z$ for $z\ge0$ gives
\[
f(v)\ge v^p\qquad(v\ge0),
\]
and, for every $r_0>0$,
\[
\operatorname{Lip}(f;[0,\rho])
=p\rho^{p-1}e^{\rho^p}
\le p e^{r_0^p}\rho^{p-1}
\qquad(0<\rho\le r_0).
\]
Thus Remark \ref{rem:general-power-blowup} and
Proposition \ref{rem:localpower} apply with exponent $p$.

For \eqref{eq:exponential-linear-nonlinearity},
differentiation gives
\begin{equation}\label{eq:exponential-linear-lip}
\operatorname{Lip}(f;[0,r_0])=cb e^{br_0},
\end{equation}
while direct integration yields
\begin{equation}\label{eq:exponential-linear-J}
J(z)=-\frac1{cb}\log(1-e^{-bz}).
\end{equation}
Consequently,
\begin{equation}\label{eq:exponential-linear-Jinverse}
J^{-1}(s)
=
-\frac1b\log(1-e^{-cb s})
\sim\frac1b e^{-cb s}
\qquad\text{as }s\to\infty.
\end{equation}
\end{remark}

Thus \eqref{eq:sum-powers} and
\eqref{eq:exponential-power-nonlinearity} are covered
by the power-type criteria, while
\eqref{eq:linear-plus-power} and
\eqref{eq:exponential-linear-nonlinearity} are governed
by the interaction between their linear parts,
the cumulative time coefficient $H$, and the spectral decay.
For the latter two nonlinearities, we use the notation
\begin{equation}\label{eq:linear-slope-notation}
a:=f'(0)=
\begin{cases}
a_0, & \text{for \eqref{eq:linear-plus-power}},\\
cb, & \text{for \eqref{eq:exponential-linear-nonlinearity}}.
\end{cases}
\end{equation}

\subsection{The role of the time coefficient}

We first consider the two nonlinearities with superlinear
behavior at the origin.

\begin{corollary}[Polynomial time coefficients]
\label{cor:beyond-polynomial-time}
Let $f$ be given by \eqref{eq:sum-powers} or
\eqref{eq:exponential-power-nonlinearity}, and let
\[
h(t)=(1+t)^q,\qquad q>-1.
\]
Assume $(HK_\beta)$ and
\[
V(x_0,R)\le CR^d
\]
for some $d>0$ and all sufficiently large $R$, with $x_0$
chosen on a ball carrying positive initial mass.
Then finite-time blow-up holds whenever
\[
1<p<1+\frac{\beta(q+1)}d.
\]
If $q=0$ and the heat kernel is conservative,
the blow-up conclusion also holds at equality.
For global existence, assume instead the standing kernel
hypotheses and
\[
\overline{\mathcal K}(t+\gamma)
\le C(1+t)^{-\kappa}
\qquad(t\ge0)
\]
for some $\kappa>0$.
Then sufficiently small initial data generate global
solutions whenever
\[
p>1+\frac{q+1}{\kappa}.
\]
If the hypotheses of Proposition \ref{prop:spectralkernel}
hold, sufficiently small data generate global solutions
for every $p>1$ and every $q>-1$.
More generally, this last conclusion holds whenever $h$
has at most polynomial growth.
\end{corollary}

\begin{proof}
The blow-up statements follow from
Remark \ref{rem:general-power-blowup} and the power results
of Section \ref{model}.
Under the polynomial upper kernel bound,
\[
I_{p,\gamma}
\le C^{p-1}
\int_0^\infty(1+t)^{q-\kappa(p-1)}\,dt<\infty
\]
when $\kappa(p-1)>q+1$.
Under the spectral estimate,
\[
I_{p,\gamma}
\le A_\gamma^{p-1}
\int_0^\infty h(t)e^{-(p-1)\lambda_1t}\,dt<\infty
\]
for every $p>1$ if $h$ has at most polynomial growth.
Proposition \ref{rem:localpower} gives global existence.
\end{proof}

\begin{corollary}[Exponentially increasing time coefficients]
\label{cor:beyond-exponential-time}
Let $f$ be given by \eqref{eq:sum-powers} or
\eqref{eq:exponential-power-nonlinearity}, and let
\[
h(t)=e^{qt},\qquad q>0.
\]
Under $(HK_\beta)$, finite-time blow-up holds for every
$p>1$.
Under the standing global existence hypotheses and those
of Proposition \ref{prop:spectralkernel}, sufficiently
small initial data generate global solutions if
\[
q<(p-1)\lambda_1.
\]
If $\lambda_1>0$ and
\eqref{eq:spectral-exponential-lower} holds, finite-time
blow-up occurs for every admissible nontrivial initial
datum if
\[
q>(p-1)\lambda_1.
\]
At equality, the additional estimate
\[
\overline{\mathcal K}(t+\gamma)
\le C_\gamma e^{-\lambda_1t}(1+t)^{-\alpha},
\qquad t\ge0,
\]
gives small-data global existence provided
$\alpha(p-1)>1$.
\end{corollary}

\begin{proof}
The blow-up conclusions follow from
Remark \ref{rem:general-power-blowup}.
For global existence,
\[
I_{p,\gamma}
\le A_\gamma^{p-1}
\int_0^\infty
e^{-((p-1)\lambda_1-q)t}\,dt<\infty
\]
when $q<(p-1)\lambda_1$.
At equality, the additional kernel estimate gives
\[
I_{p,\gamma}
\le C_\gamma^{p-1}
\int_0^\infty(1+t)^{-\alpha(p-1)}\,dt<\infty
\]
when $\alpha(p-1)>1$.
Apply Proposition \ref{rem:localpower}.
\end{proof}

We next turn to the nonlinearities with a positive linear
part. Their inverse functions $J^{-1}$ have exponential
decay, as shown in
\eqref{eq:linear-power-Jinverse} and
\eqref{eq:exponential-linear-Jinverse}.

\begin{proposition}[Time coefficients in the presence of a linear part]
\label{prop:linear-part-time}
Let $f$ be given by \eqref{eq:linear-plus-power} or
\eqref{eq:exponential-linear-nonlinearity}, and let $a>0$
be defined by \eqref{eq:linear-slope-notation}.

\begin{enumerate}
\item
Suppose that, on every relatively compact ball $B$,
there exist $c_B>0$, $\kappa_B>0$, and $T_B>0$ such that
\[
K_T(x,y)\ge c_BT^{-\kappa_B}
\]
for every $T\ge T_B$ and almost every $(x,y)\in B\times B$.
Then finite-time blow-up holds if
\[
h(t)=(1+t)^q,\quad q>-1,
\qquad\text{or}\qquad
h(t)=e^{qt},\quad q>0.
\]

\item
Under the standing global existence hypotheses and
the hypotheses of Proposition \ref{prop:spectralkernel},
sufficiently small initial data generate global solutions
if
\[
h\equiv1,\qquad a<\lambda_1,
\]
or if
\[
h(t)=(1+t)^q,\qquad -1<q<0.
\]
In the latter case, no restriction on $a$ relative
to $\lambda_1$ is required.

\item
Suppose that $\lambda_1>0$ and
\eqref{eq:spectral-exponential-lower} holds.
Then finite-time blow-up occurs if
\[
h\equiv1,\qquad a>\lambda_1.
\]
It also occurs, for every $a>0$, if
\[
h(t)=(1+t)^q,\quad q>0,
\qquad\text{or}\qquad
h(t)=e^{qt},\quad q>0.
\]
\end{enumerate}
The autonomous borderline $a=\lambda_1$ is not settled
by these strict criteria.
\end{proposition}

\begin{proof}
For both nonlinearities,
\[
J^{-1}(s)\sim C_f e^{-as}
\qquad\text{as }s\to\infty
\]
for some $C_f>0$.

For part (i), choose a relatively compact ball carrying
positive initial mass.
For both time coefficients in that part,
\[
\frac{H(T)}{\log T}\longrightarrow\infty.
\]
The polynomial lower heat kernel bound therefore gives
\[
\frac{\underline{\mathcal K}(T)}{J^{-1}(H(T))}
\longrightarrow\infty.
\]
Corollary \ref{corteo1} yields blow-up.

For part (ii), first suppose that $h\equiv1$ and
$a<\lambda_1$.
Choose
\[
a<\ell<\nu<\lambda_1.
\]
By \eqref{eq:linear-power-lip} or
\eqref{eq:exponential-linear-lip}, a sufficiently small
$r_0>0$ ensures \eqref{locallipschitzgap}.
Since $H(t)=t$, condition \eqref{eq:general-time-gap}
holds with $B_\nu=1$.
Theorem \ref{thm:gap-general} applies.

If $-1<q<0$, then $H(t)=o(t)$.
Choose any $\ell>a$, then $r_0>0$ sufficiently small
for \eqref{locallipschitzgap}, and finally $\nu>\ell$.
Since
\[
\nu H(t)-\lambda_1t\longrightarrow-\infty,
\]
condition \eqref{eq:general-time-gap} holds.
Again, Theorem \ref{thm:gap-general} applies.

For part (iii), the spectral lower bound gives
\eqref{hpBU} whenever, for some $\eta>0$,
\[
aH(T)-(\lambda_1+\eta)T\longrightarrow\infty.
\]
If $h\equiv1$ and $a>\lambda_1$, choose
$0<\eta<a-\lambda_1$.
For the increasing polynomial and exponential coefficients,
the conclusion follows from $H(T)/T\to\infty$.
\end{proof}

\begin{remark}
The sufficient conditions above distinguish the behavior
near zero from the growth at infinity.
For \eqref{eq:sum-powers}, the smaller exponent $p$ determines
the displayed thresholds; the larger exponent $r$ affects
the quantitative smallness condition.
For \eqref{eq:exponential-power-nonlinearity}, exponential
growth at infinity is likewise compatible with small-data
global existence because the construction uses only
local control near zero.

For \eqref{eq:linear-plus-power} and
\eqref{eq:exponential-linear-nonlinearity}, the positive
slope $a=f'(0)$ changes the comparison with the time
coefficient and the spectral decay.

All four nonlinearities satisfy \eqref{Hf}.
Thus, for every time coefficient considered here,
sufficiently large multiples of a fixed nonnegative profile
with $\|P_T\varphi\|_{L^\infty}>0$ at some $T>0$
generate blow-up by Remark \ref{rem:large-data-BU},
including in the regimes admitting small-data global solutions.
\end{remark}

\subsection{Examples of underlying spaces}

We now apply the preceding criteria to the geometric
settings considered in Section \ref{model}.
The distinction between the two nonlinearities with
superlinear behavior at zero and the two with a positive
linear part will be kept explicit.

\subsubsection{Euclidean space and Riemannian manifolds}
\begin{example}[Euclidean space and flat cylinders]
Let $X=\mathbb R^n$.
\begin{itemize}
\item For either \eqref{eq:sum-powers} or
\eqref{eq:exponential-power-nonlinearity}, with $h\equiv1$,
finite-time blow-up holds whenever
\[
1<p\le1+\frac2n,
\]
while sufficiently small initial data generate global
solutions whenever
\[
p>1+\frac2n.
\]
\item For
\[
h(t)=(1+t)^q,\qquad q>-1,
\]
the corresponding strict conditions are
\[\begin{aligned}
&1<p<1+\frac{2(q+1)}n
\quad\text{for blow-up},\\
&p>1+\frac{2(q+1)}n
\quad\text{for small-data global existence}.\end{aligned}
\]
The critical equality is included in the blow-up range
when $q=0$; the preceding criteria do not settle it
for $q\ne0$.
\item For $h(t)=e^{qt}$, $q>0$, blow-up holds for every $p>1$.
\end{itemize}
For either \eqref{eq:linear-plus-power} or
\eqref{eq:exponential-linear-nonlinearity}, blow-up holds
for every admissible nontrivial initial datum for all
the polynomial and exponential coefficients just listed.
Indeed, the Euclidean kernel satisfies the local polynomial
lower bound required by
Proposition \ref{prop:linear-part-time}.

On $\mathbb R^d\times\mathbb T^m$, with $d\ge1$ and $m\ge1$, the same conclusions hold
with $d$ in place of $n$.
For the nonlinearities with superlinear behavior at zero,
the threshold is therefore determined by the large-scale
dimension $d$, independently of the compact factor.
\end{example}

\begin{example}[Nonnegative Ricci curvature]
Let $X$ be a complete connected $n$-dimensional Riemannian
manifold with nonnegative Ricci curvature.

For either \eqref{eq:sum-powers} or
\eqref{eq:exponential-power-nonlinearity}, one has
\begin{itemize}
\item for \[
h(t)=(1+t)^q,\qquad q>-1,
\]
that the Gaussian estimates and the volume upper bound $V(x_0,R)\le CR^n$ give blow-up whenever
\[
1<p<1+\frac{2(q+1)}n.
\]
If in particular $q=0$, conservativity also gives blow-up at equality.
If, in addition,
\[
V(x,R)\ge cR^d
\qquad\text{uniformly in }x,\quad R\ge1,
\]
then the uniform upper heat kernel bounds discussed in
Section \ref{model} yield global existence for sufficiently
small data whenever
\[
p>1+\frac{2(q+1)}d.
\]
If the large-scale volume growth is uniformly comparable
to $R^d$, the blow-up threshold also involves $d$, so the
two strict criteria have the common threshold
\[
1+\frac{2(q+1)}d.
\]
Again, equality belongs to the blow-up range when $q=0$.
\item For $h(t)=e^{qt}$, $q>0$, both nonlinearities yield blow-up
for every $p>1$, without an additional volume lower bound.
\end{itemize}
For either \eqref{eq:linear-plus-power} or
\eqref{eq:exponential-linear-nonlinearity}, all the above
polynomial and exponential coefficients yield blow-up.
Indeed, the Gaussian lower bound and
$V(x,R)\le CR^n$ give a polynomial lower kernel bound
on every fixed relatively compact ball, so
Proposition \ref{prop:linear-part-time} applies.
\end{example}

\begin{example}[Hyperbolic space]
Let $X=\mathbb H^n$, $n\ge2$, with sectional curvature $-1$.
Its heat kernel is uniformly bounded at every positive time,
and
\[
\lambda_1=\frac{(n-1)^2}{4}>0.
\]
The locally uniform large-time heat kernel asymptotics
recalled in Section \ref{mf} also give
\eqref{eq:spectral-exponential-lower}; see
\cite{ChavelKarp}.
For either \eqref{eq:sum-powers} or
\eqref{eq:exponential-power-nonlinearity},
sufficiently small initial data generate global solutions
for every $p>1$ whenever $h$ has at most polynomial growth.
In particular, this holds for every
\[
h(t)=(1+t)^q,\qquad q>-1.
\]
For $h(t)=e^{qt}$, $q>0$, the spectral criteria give
\[
q<(p-1)\frac{(n-1)^2}{4}
\quad\text{for small-data global existence},
\]
and
\[
q>(p-1)\frac{(n-1)^2}{4}
\quad\text{for blow-up of every admissible nontrivial datum}.
\]

For \eqref{eq:linear-plus-power}, with $h\equiv1$,
sufficiently small data generate global solutions if
\[
a_0<\frac{(n-1)^2}{4},
\]
while every admissible nontrivial datum generates blow-up if
\[
a_0>\frac{(n-1)^2}{4}.
\]
For \eqref{eq:exponential-linear-nonlinearity}, the corresponding
conditions are
\[
cb<\frac{(n-1)^2}{4}
\quad\text{and}\quad
cb>\frac{(n-1)^2}{4},
\]
respectively.

For both nonlinearities with a positive linear part,
the decreasing coefficients
\[
h(t)=(1+t)^q,\qquad -1<q<0,
\]
give global existence for sufficiently small data,
without a restriction on $a_0$ or $cb$ relative
to $\lambda_1$.
In contrast, the increasing coefficients
\[
h(t)=(1+t)^q,\quad q>0,
\qquad\text{or}\qquad
h(t)=e^{qt},\quad q>0,
\]
give blow-up for every admissible nontrivial initial datum.
The equality cases are subject to the qualifications
in Corollary \ref{cor:beyond-exponential-time} and
Proposition \ref{prop:linear-part-time}.
\end{example}

\begin{example}[Negatively curved manifolds and other spectral settings]
Let $X$ be a complete simply connected Riemannian manifold
of dimension $n\ge2$ such that
\[
-b_1^2\le\operatorname{Sec}_X\le-a_1^2<0
\]
for some $0<a_1\le b_1$.
As in Section \ref{model}, its heat kernel has uniform
bounds at positive times and
\[
\lambda_1(X)\ge\frac{(n-1)^2a_1^2}{4}>0.
\]
By \cite[Theorem~3.1(b) and Section~5.1]{KLVW2015},
\[
\lim_{t\to\infty}\frac{\log K_t(x,y)}{t}
=-\lambda_1(X)
\qquad\text{for every }x,y\in X.
\]
Note that, for every relatively compact ball
$B\subset X$ and every $\eta>0$,
\[
K_t(x,y)\ge c_{B,\eta}e^{-(\lambda_1(X)+\eta)t}
\qquad
(x,y\in B,\ t\ge T_{B,\eta}),
\]
for suitable constants $c_{B,\eta},T_{B,\eta}>0$ (see \cite{GrilloMeglioliPunzo2023}).
Thus \eqref{eq:spectral-exponential-lower} holds.
All the conclusions of the hyperbolic example therefore
hold with $\lambda_1(X)$ in place of $(n-1)^2/4$.
The lower bound on $\lambda_1(X)$ also gives explicit
sufficient conditions for existence:
\[
q<(p-1)\frac{(n-1)^2a_1^2}{4}
\]
for \eqref{eq:sum-powers} and
\eqref{eq:exponential-power-nonlinearity} with $h(t)=e^{qt}$,
and
\[
a_0<\frac{(n-1)^2a_1^2}{4},
\qquad
cb<\frac{(n-1)^2a_1^2}{4},
\]
for the respective nonlinearities with a linear part
and $h\equiv1$.
The blow-up comparisons involve the actual value
$\lambda_1(X)$; a lower bound for it alone does not
give the corresponding sufficient blow-up inequalities.

More generally, the small-data existence conclusions apply
on spaces satisfying the standing kernel assumptions,
with positive spectral bottom and a uniform kernel bound
at a positive time. Similar results can be found in \cite{GrilloMeglioliPunzo2023}.
\end{example}

\subsubsection{Metric and discrete graphs}

We retain the graph realizations, measures, and Laplacians
specified in Examples \ref{ex:metric-lattice},
\ref{ex:discrete-lattice}, and \ref{ex:regular-tree}.
For discrete graphs, we use the discrete semigroup framework
described in Section \ref{model}, rather than imposing
the all-time estimate $(HK_\beta)$.
Small-data statements retain the heat kernel majorization
specified at the beginning of this section.

\begin{example}[Metric and discrete lattices]
Let $X$ be either the metric lattice of
Example \ref{ex:metric-lattice} or the discrete lattice
of Example \ref{ex:discrete-lattice}, with $d\ge1$.

First, let $f$ be given by \eqref{eq:sum-powers} or
\eqref{eq:exponential-power-nonlinearity}.
For
\[
h(t)=(1+t)^q,\qquad q>-1,
\]
every admissible nonnegative, nontrivial initial datum
generates finite-time blow-up if
\[
1<p<1+\frac{2(q+1)}d,
\]
whereas sufficiently small initial data generate global
solutions if
\[
p>1+\frac{2(q+1)}d.
\]
When $q=0$, blow-up also holds at equality.
For $q\ne0$, the preceding criteria do not settle
the equality case.
If
\[
h(t)=e^{qt},\qquad q>0,
\]
finite-time blow-up holds for every $p>1$.

For the metric lattice, these statements follow from
the heat kernel estimates in Example \ref{ex:metric-lattice}
and the power-type criteria of this section.
For the discrete lattice, the subcritical and
exponential-time blow-up statements use the lower bound
on a fixed finite set, while global existence follows from
\eqref{eq:discrete-lattice-upper} and
Proposition \ref{rem:localpower}.
At the autonomous critical exponent, the discrete mass
argument in Example \ref{ex:discrete-lattice} remains valid:
the inequality $f(v)\ge v^p$ gives the same restarted
bound and the same logarithmic contradiction.

Next, let $f$ be given by \eqref{eq:linear-plus-power} or
\eqref{eq:exponential-linear-nonlinearity}.
Then every admissible nonnegative, nontrivial initial datum
generates finite-time blow-up for each of the coefficients
\[
h(t)=(1+t)^q,\quad q>-1,
\qquad\text{and}\qquad
h(t)=e^{qt},\quad q>0.
\]
Indeed, both lattices satisfy a polynomial lower heat
kernel bound on each fixed relatively compact ball
(a finite ball in the discrete case).
The argument of Proposition \ref{prop:linear-part-time}(i)
therefore applies.

For the nonlinearities with power behavior at zero,
the displayed thresholds depend on the large-scale
dimension $d$.
In particular, the local one-dimensional structure
of the metric lattice does not change these thresholds.
\end{example}

\begin{example}[Regular discrete trees]
Let $X=\mathbb T_k$, $k\ge3$, with counting measure
and the unnormalized Laplacian of
Example \ref{ex:regular-tree}. Set
\[
\lambda_1=k-2\sqrt{k-1}>0.
\]
The uniform kernel bound and the lower estimate
on fixed finite balls established in that example
allow us to apply the spectral criteria of this section.

Let $f$ be given by \eqref{eq:sum-powers} or
\eqref{eq:exponential-power-nonlinearity}.
If $h$ has at most polynomial growth, sufficiently small
initial data generate global solutions for every $p>1$.
In particular, this holds for
\[
h(t)=(1+t)^q,\qquad q>-1.
\]
For $h(t)=e^{qt}$, $q>0$, sufficiently small data
generate global solutions when
\[
q<(p-1)\lambda_1,
\]
whereas every admissible nonnegative, nontrivial datum
generates finite-time blow-up when
\[
q>(p-1)\lambda_1.
\]
These strict conditions do not settle equality.

Now let $f$ be given by \eqref{eq:linear-plus-power} or
\eqref{eq:exponential-linear-nonlinearity}, and write
\[
a=f'(0)=a_0
\quad\text{or}\quad
a=f'(0)=cb,
\]
respectively.
For $h\equiv1$, sufficiently small data generate global
solutions if $a<\lambda_1$, whereas every admissible
nonnegative, nontrivial datum generates finite-time
blow-up if $a>\lambda_1$.
The autonomous equality case is not settled
by these strict criteria.

For either nonlinearity with a positive linear part,
the coefficient
\[
h(t)=(1+t)^q,\qquad -1<q<0,
\]
gives global existence for sufficiently small data,
without a restriction on $a$ relative to $\lambda_1$.
For
\[
h(t)=(1+t)^q,\quad q>0,
\qquad\text{or}\qquad
h(t)=e^{qt},\quad q>0,
\]
every admissible nonnegative, nontrivial datum
generates finite-time blow-up.

Finally, for all four nonlinearities, sufficiently
large multiples of a fixed admissible nonnegative,
nontrivial profile generate finite-time blow-up,
including in the regimes admitting small-data
global solutions.
\end{example}

\begin{remark}
The autonomous small-data criterion for a positive linear
part is not restricted to the explicit families considered
here.
If $f$ satisfies the standing assumptions and is convex,
then
\[
f'_+(0)<\lambda_1
\]
ensures \eqref{locallipschitzgap} with some $\ell<\lambda_1$
and sufficiently small $r_0>0$.
Indeed, nonnegativity and $f(0)=0$ imply that $f$ is
nondecreasing, and
\[
\operatorname{Lip}(f;[0,r_0])=f'_-(r_0),
\qquad
\lim_{r_0\downarrow0}f'_-(r_0)=f'_+(0).
\]
Theorem \ref{thm:gap-general} then applies with $h\equiv1$
under its remaining kernel and initial-data assumptions.
\end{remark}

\bigskip
\bigskip

\appendix
\section{Additional background on heat kernel estimates}
\label{app:kernel-background}
\setcounter{equation}{0}

This appendix collects supplementary results relating
heat kernel profiles to volume growth and locality
of the Dirichlet form.
We retain the metric measure and Dirichlet form
framework introduced in Section \ref{mf}.

\subsection{Power prefactors and volume growth}

We next introduce estimates whose prefactor is a power of time, rather than
the reciprocal of the volume of a ball.

\begin{definition}\label{defHtheta}
Fix $\alpha,\beta>0$. A heat kernel satisfies $\mathcal H_\theta$, where $\theta>0$,
if there exist finite-valued, nonnegative, nonincreasing functions
$\Phi_1,\Phi_2:[0,\infty)\to[0,\infty)$ such that $\Phi_1(1)>0$,
\begin{equation}\label{Htheta}
\frac{1}{t^{\alpha/\beta}}
\Phi_1\left(\frac{d(x,y)}{t^{1/\beta}}\right)
\le K_t(x,y)\le
\frac{1}{t^{\alpha/\beta}}
\Phi_2\left(\frac{d(x,y)}{t^{1/\beta}}\right)
\end{equation}
for every $t>0$ and almost every $(x,y)$, and
\begin{equation}\label{eqint}
\int_1^\infty s^{\theta-1}\Phi_2(s)\,ds<\infty.
\end{equation}
\end{definition}

For fixed $\alpha,\beta$, the condition $\mathcal H_\theta$ becomes stronger
as $\theta$ increases. For example,
\[
\Phi_2(s)=e^{-Cs^\gamma},
\qquad C,\gamma>0,
\]
satisfies \eqref{eqint} for every $\theta>0$, whereas
\[
\Phi_2(s)=\frac{C}{(1+s)^\gamma},
\qquad C,\gamma>0,
\]
satisfies \eqref{eqint} precisely when $0<\theta<\gamma$.

The Gauss--Weierstrass heat kernel on $\mathbb R^n$,
\[
p_t(x,y)=
\frac{1}{(4\pi t)^{n/2}}\exp\left(-\frac{|x-y|^2}{4t}\right),
\]
satisfies $\mathcal H_\theta$ for every $\theta>0$, with $\alpha=n$, $\beta=2$, and
\[
\Phi_1(s)=\Phi_2(s)=
\frac{1}{(4\pi)^{n/2}}\exp\left(-\frac{s^2}{4}\right).
\]

Since $\Phi_2$ is nonincreasing and finite-valued, we have
\[
\int_0^1 s^{\alpha-1}\Phi_2(s)\,ds
\le \frac{\Phi_2(0)}{\alpha}<\infty.
\]
Thus, when $\theta=\alpha$, condition \eqref{eqint} is equivalent to
\[
\int_0^\infty s^{\alpha-1}\Phi_2(s)\,ds<\infty.
\]

\begin{theorem}[see \cite{GrigHK,GrigHuLau}]
Let $K_t$ be a conservative heat kernel on $(X,d,\mu)$.
Assume that $\mathcal H_\alpha$ holds. Then there exists a constant $c\ge1$ such that
\begin{equation}\label{eq:polbound}
c^{-1}r^\alpha\le V(x,r)\le cr^\alpha
\end{equation}
for every $x\in X$ and every $r>0$.
\end{theorem}

\begin{remark}
The estimate \eqref{eq:polbound} means that $(X,d,\mu)$ is
Ahlfors $\alpha$-regular at all scales.

The exponent $\alpha$ in \eqref{eq:polbound} is unique and is determined by
the metric measure space. In particular, under conservativity and $\mathcal H_\alpha$,
it is an invariant of $(X,d,\mu)$.

Since \eqref{eq:polbound} holds at every scale and $\alpha>0$, it implies that $\mu$
has no atoms and that $\mu(X)=\infty$. Thus, these hypotheses exclude compact spaces
and discrete spaces with positive point masses.
\end{remark}

\begin{remark}
If $V(x,r)\asymp r^\alpha$ uniformly in $x\in X$ and $r>0$,
then $(HK_\beta)$ yields two-sided heat kernel bounds with
prefactor $t^{-\alpha/\beta}$ and sub-Gaussian profiles.
In particular, $\mathcal H_\theta$ holds for every $\theta>0$.
\end{remark}

\subsection{The chain condition and locality}

\begin{definition}\label{chaincondition}
A metric space $(X,d)$ satisfies the \emph{chain condition} if there exists $C\ge1$
such that, for every $x,y\in X$ and every integer $m\ge1$, there are points
$x_0,\ldots,x_m\in X$ satisfying
\[
x_0=x,\qquad x_m=y,
\]
and
\[
d(x_i,x_{i+1})\le \frac{C}{m}d(x,y),
\qquad i=0,\ldots,m-1.
\]
\end{definition}

\begin{remark}
Under Definition \ref{MMS}, the chain condition is automatic, with $C=1$,
by subdividing a geodesic into $m$ segments of equal length.
Moreover, all metric balls are relatively compact.

More generally, the chain condition holds on quasiconvex metric spaces, meaning spaces
in which every pair of points can be joined by a rectifiable curve whose length
is bounded by a fixed multiple of their distance.
\end{remark}

\begin{theorem}[see \cite{GrigKum}]\label{teolocalandnonlocal}
Let $(X,d,\mu)$ be as in Definition \ref{MMS}. Let $K_t$ be a conservative heat kernel
whose associated Dirichlet form $(\mathcal E,\mathcal F)$ is regular.

Assume that there exist $\alpha,\beta>0$, positive constants $a_1,a_2,a_3,a_4$,
and a finite-valued, nonnegative, nonincreasing function
$\Phi:[0,\infty)\to[0,\infty)$, with $\Phi(s_0)>0$ for some $s_0>0$, such that
\[
\frac{a_1}{t^{\alpha/\beta}}
\Phi\left(a_2\frac{d(x,y)}{t^{1/\beta}}\right)
\le K_t(x,y)\le
\frac{a_3}{t^{\alpha/\beta}}
\Phi\left(a_4\frac{d(x,y)}{t^{1/\beta}}\right)
\]
for every $t>0$ and almost every $(x,y)$.

Then
\[
\beta\le\alpha+1,
\qquad
V(x,r)\asymp r^\alpha
\quad \text{uniformly in }x\in X,\ r>0.
\]
Moreover, exactly one of the following alternatives holds:
\begin{itemize}
\item the form $\mathcal E$ is local, $\beta\ge2$, and there exist $b_1,b_2,b_3,b_4>0$
such that
\[
b_1\exp\left(-b_2s^{\beta/(\beta-1)}\right)
\le\Phi(s)\le
b_3\exp\left(-b_4s^{\beta/(\beta-1)}\right)
\qquad \text{for every }s\ge0;
\]

\item the form $\mathcal E$ is non-local and
\[
\Phi(s)\asymp(1+s)^{-(\alpha+\beta)}
\qquad \text{for every }s\ge0.
\]
\end{itemize}
\end{theorem}

\begin{remark}
The theorem requires a single profile $\Phi$, up to multiplicative constants
and rescaling of its argument. The bounds in \eqref{Htheta} with two unrelated
profiles $\Phi_1,\Phi_2$ do not by themselves provide this hypothesis.

The chain condition and the relative compactness of balls required in \cite{GrigKum}
are ensured here by Definition \ref{MMS}.
\end{remark}

\noindent{\bf Acknowledgement}

\noindent{\bf Conflict of interest}. The author states no conflict of interest.

\noindent{\bf Data availability statement}. There are no data associated with this research.

\end{document}